\documentclass[12pt, reqno]{amsart}
\advance\oddsidemargin-2cm
\advance\evensidemargin-2cm
\usepackage{amsmath}
\usepackage{amssymb}

\theoremstyle{plain}
\newtheorem{thm}{Theorem}[section]
\newtheorem{lem}[thm]{Lemma}
\newtheorem{rem}[thm]{Remark}
\newtheorem{prop}[thm]{Proposition}

\newtheorem{coro}[thm]{Corollary}

\numberwithin{equation}{section}
\allowdisplaybreaks[1]

\theoremstyle{definition}

\newtheorem{defn}[thm]{Definition}

\theoremstyle{remark}

\newcommand{\bn}{\mathbb{B}_n}
\newcommand{\CC}{\mathbb{C}}
\newcommand{\cn}{\mathbb{C}^n}
\newcommand{\ber}{L_a^2}
\newcommand{\poly}{\mathbb{C}[z_1,\cdots,z_n]}
\newcommand{\la}{\langle}
\newcommand{\ra}{\rangle}

\newcommand{\Z}{\mathcal{Z}}
\newcommand{\LL}{\mathcal{L}}
\newcommand{\ds}{d_S}

\newcommand{\Bd}{\mathbb{B}_{\delta}}
\newcommand{\pij}{p_{ij}}
\newcommand{\fij}{f_{ij}}
\newcommand{\gij}{g_{ij}}
\newcommand{\dist}{\mathrm{dist}}

\newcommand{\PP}{\mathcal{P}}
\newcommand{\QQ}{\mathcal{Q}}

\newcommand{\E}{\tilde{E}}
\newcommand{\F}{\tilde{F}}
\newcommand{\PPP}{\hat{P}}
\newcommand{\W}{\mathcal{W}}
\newcommand{\T}{\hat{T}}

\begin{document}
	
	\title[]{Essential Normality - a unified Approach in Terms of Local Decompositions}
	\author[]{Yi Wang}
	\address{Department of Mathematics, State University of New York at Buffalo, Buffalo, NY 14260, USA}
	\email{yiwangfdu@gmail.com}
	
	
		\maketitle
	
	\centerline{\it Dedicated to the memory of Ronald G. Douglas}

	\begin{abstract}
		In this paper, we define the asymptotic stable division property for submodules of $\ber(\bn)$. We show that under a mild condition, a submodule with the asymptotic stable division property is $p$-essentially normal for all $p>n$. A new technique is developed to show that certain submodules have the asymptotic stable division property. This leads to a unified proof of most known results on essential normality of submodules as well as new results. In particular, we show that an ideal defines a $p$-essentially normal submodule of $\ber(\bn)$, $\forall p>n$, if its associated primary ideals are powers of prime ideals whose zero loci satisfy standard regularity conditions near the sphere.
	\end{abstract}

	\section{Introduction}
	\setcounter{equation}{0}
	Let $\bn$ be the open unit ball in $\cn$. The Bergman space $\ber(\bn)$ consists of all holomorphic functions $f$ on $\bn$ such that 
	\[
	\|f\|^2=\int_{\bn}|f|^2dv<\infty.
	\]
	Here $v$ denotes the normalized Lebesgue measure, i.e., $v(\bn)=1$. For $i=1,\cdots,n$, the coordinate functions $z_i$ acts on $\ber(\bn)$ by multiplication:
	\[
	M_{z_i}f=z_if,\quad f\in\ber(\bn).
	\]
	The $n$-tuple of operators $(M_{z_1},\cdots,M_{z_n})$ are commuting and thus induces a Hilbert $\poly$-module structure on $\ber(\bn)$:
	\[
	\poly\times\ber(\bn)\to\ber(\bn),\quad (p,f)\mapsto p(M_{z_1},\cdots,M_{z_n})f=pf.
	\]
	For any $i, j=1,\cdots,n$, it is well known that the commutator $[M_{z_i}, M_{z_j}^*]$ belongs to the Schatten class $\mathcal{C}_p$, $\forall p>n$.
	
	A closed subspace $\PP\subset\ber(\bn)$ that is invariant under $M_{z_i}$, $i=1,\cdots,n$, is called a (Hilbert) \emph{submodule} of $\ber(\bn)$. The commuting tuple $(R_1,\cdots,R_n)$, where $R_i=M_{z_i}|_{\PP}$, defines the module action on $\PP$. Its orthogonal complement $\QQ:=\PP^\perp$ is called a \emph{quotient module} of $\ber(\bn)$. The module action on $\QQ$ is defined by the tuple $(S_1,\cdots,S_n)$, where $S_i=QM_{z_i}|_{\QQ}$. Here $Q$ denotes the projection operator onto $\QQ$. For $p\geq1$, we say $\PP$ ($\QQ$) is $p$-essentially normal if $[R_i, R_j^*]\in\mathcal{C}_p$ ($[S_i, S_j^*]\in\mathcal{C}_p$).
	
	For an ideal $I$ in $\poly$, let $\PP_I$ denote its closure in $\ber(\bn)$. Then it is easy to see that $\PP_I$ is a submodule of $\ber(\bn)$. Therefore $\QQ_I:=\PP_I^\perp$ is a quotient module.
	
	\medskip

	\noindent\textbf{Arveson-Douglas Conjecture:} Suppose $I$ is a homogeneous ideal in $\poly$. Then the quotient module $\QQ_I$ is $p$-essentially normal for all $p>\dim_{\CC}Z(I)$.
	
	\begin{rem}
		The Arveson-Douglas Conjecture was originally stated on the Drury-Arveson space $H_n^2$. Later it was shown that, for a homogeneous ideal $I$ and any $p>n$, the $p$-essential normalities for the closures of $I$ in the Drury-Arveson space, the Bergman space and the Hardy space, are equivalent. Closures of non-homogeneous ideals and non-polynomial generated submodules are also considered . Submodules on other domains were also considered.
		
		In this paper, we consider submodules in the Bergman module. We will consider $p$-essential normality for $p>n$. For $p>n$, the $p$-essential normality of a submodule $\PP$ is equivalent to the $p$-essential normality of its quotient module $\QQ$.
	\end{rem}
	
	The Arveson-Douglas Conjecture arises from Arveson's study of row contractions in multivariable operator theory \cite{Arv98}-\cite{Arv07}. Later, Douglas \cite{Dou06} showed that, given an essentially normal quotient module $\QQ_I$, the short exact sequence
	\[
	0\to\mathcal{K}\to C^*(\{S_i\}, I)+\mathcal{K}\to C(X_I)\to0
	\]
	defines an element $[\QQ_I]$ in the odd K-homology group $K_1(X_I)$ of a topological space $X_I$. One can show that $\overline{\bn\cap Z(I)}\cap\partial\bn\subset X_I\subset\partial\bn\cap Z(I)$. In the case $I$ is homogeneous, $X_I=Z(I)\cap\partial\bn$. The element $[\QQ_I]$ carries geometric information of $Z(I)$. This gives a new kind of index theorem. Moreover, a positive result of the Arveson-Douglas Conjecture will lead to an analytic Grothendieck-Riemann-Roch theorem allowing singularities \cite{DTY}.
	
	The existing results of the Arveson-Douglas Conjecture can be roughly categorized into three types. The first type contains results concerning varieties of dimension $1$, or codimension $1$. In \cite{GWK}, Kuo and Wang proved the cases of homogeneous ideals $I$ when $n\leq3$, or $\dim Z(I)\leq1$, or $I$ is principal. Douglas and Wang \cite{DWK} showed that for a principal ideal $I$, not necessarily homogeneous, $\PP_I\subset\ber(\bn)$ is $p$-essentially normal for all $p>n$. Fang and Xia \cite{FX13}\cite{FX18} extended the results to polynomial-generated principal submodules of the Hardy space $H^2(\bn)$, and, under additional assumptions, the Drury-Arveson space $H_n^2$. Douglas, Guo and the author \cite{DGW} showed that a principal submodule of $\ber(\Omega)$ generated by a function $h\in\mathrm{Hol}(\overline{\Omega})$ is $p$-essentially normal for all $p>n$. Here $\Omega\subset\cn$ is any bounded strongly pseudoconvex domain with smooth boundary. The result was extended to the Hardy spaces $H^2(\Omega)$ by the author and Xia \cite{WX2}.
	
	The second type of results concern a geometric version of the Arveson-Douglas Conjecture. The results involve varieties with geometric conditions such as smoothness and transversallity on $\partial\bn$. Engli\v{s} and Eschmeier \cite{EE} showed that, if a variety $V$ is homogeneous and its only possible singular point is the origin, then the radical ideal $I$ of all polynomials vanishing on $V$ defines a $p$-essentially normal quotient module for any $p>\dim V$. Douglas, Tang and Yu \cite{DTY} showed that, if $I$ is radical and $Z(I)$ is a complete intersection space that is smooth on $\partial\bn$, intersects transversely with $\partial\bn$, then $\QQ_I$ is essentially normal. Douglas and the author \cite{DWY1} showed that, if $I$ is a radical ideal and $Z(I)$ is smooth on $\partial\bn$ and intersects transversely with $\partial\bn$, then $\QQ_I$ is $p$-essentially normal for all $p>2\dim Z(I)$. The result was then refined to all $p>\dim Z(I)$ by the author and Xia \cite{WX1}.
	
	The third type of results involve conditions that ensure decompositions of the submodules, or quotient modules into nice parts \cite{KS}\cite{Sha}\cite{DWY2}. In particular, in \cite{Sha}, Shalit considered the stable division property of a submodule in $H^2_n$ and showed that a graded submodule with the stable division property is $p$-essentially normal for all $p>n$.
	
	The aim of this paper is to provide a unified proof of most of the known Bergman-space results above. We define the asymptotic stable division property (Definition \ref{defn: ASD}) and show that the asymptotic stable division property leads to essential normality. Our first main result is the following.
	
	\begin{thm}[Theorem \ref{thm: ASD to EN}]
		Suppose $\PP$ is a submodule of $\ber(\bn)$ with the asymptotic stable division property. If the generating functions $h_i$ are all defined in a neighborhood of $\overline{\bn}$, and the controlling constants $C_i$, $N_i$, determined by $h_i$ (as in Theorem \ref{thm: key inequality}), are uniformly bounded for all $i\in\Lambda$, then the submodule $\PP$ is $p$-essentially normal for all $p>n$. In particular, if the generating functions $h_i$ are polynomials of uniformly bounded degrees, then $\PP$ is $p$-essentially normal for all $p>n$.
	\end{thm}
The proof of Theorem \ref{thm: ASD to EN} involves an inequality of a new type (Theorem \ref{thm: key inequality}) that first appeared in \cite{DGW}. Since principal submodules and graded submodules with the stable division property have the asymptotic stable division property trivially, Theorem \ref{thm: ASD to EN} provides a unified proof for the two types of results immediately.

	We will also introduce technical hypotheses (Hypothesis 1) that lead to the asymptotic stable division property (Theorem \ref{thm: ASD}). Then we will prove our second main result.
	\begin{thm}[Theorem \ref{thm: Smooth implies ASD}]\label{thm: non-radical}
		Suppose $I$ is an ideal in $\poly$ with primary decomposition $I=\cap_{j=1}^k I_j^{m_j}$, where $I_j$ are prime ideals. Assume the following.
		\begin{itemize}
			\item[(1)] For each $j=1,\cdots,k$, $Z(I_j)$ has no singular points on $\partial\bn$ and intersects $\partial\bn$ transversely.
			\item[(2)] Any pair of the varieties $\{Z(I_j)\}$ does not intersect on $\partial\bn$. 
		\end{itemize}
		Then the submodule $\PP_I$ has the asymptotic stable division property with generating elements $\{h_i\}$ being polynomials of uniformly bounded degrees. As a consequence, $\PP_I$ is $p$-essentially normal for all $p>n$.
	\end{thm}
	
	In \cite[5.2]{DTY}, the authors mentioned a plan of studying non-radical ideals. Theorem \ref{thm: non-radical} partially accomplishes this goal, with a different approach.
	
	Theorem \ref{thm: Smooth implies ASD} shows that, results of the second type also fit into this framework. The proof of Theorem \ref{thm: Smooth implies ASD} combines several techniques. First, we construct a covering that satisfies the bounded overlap condition for Bergman neighborhoods with large radius. The construction involves a radial-spherical decomposition method in \cite{Xia18}. Then we construct a decomposition formula for each covering set. The generating functions are modified from local canonical defining functions of the variety. Combining these techniques, we show that the ideals in Theorem \ref{thm: Smooth implies ASD} satisfy Hypothesis 1, and therefore are $p$-essentially normal for all $p>n$.
	
	In Section \ref{section: preliminary}, we provide some tools that will be used in this paper. In Section 3 we introduce the asymptotic stable division property, and give a proof of Theorem \ref{thm: ASD to EN}. In Section 4 to Section 6, we prove Theorem \ref{thm: Smooth implies ASD}. In the concluding remarks, we describe our future plans. In the Appendix, we prove some results involving algebraic sets. These results will be used mainly in Section \ref{section: local decomposition formulas}.
	
	\medskip
	\noindent\textbf{Acknowledgment:} The author would like to express her very great appreciation to Ronald G. Douglas in Texas A\& M University for the inspiring discussions and many support. She would also like to offer her special thanks to Jingbo Xia in SUNY Buffalo, for reading a draft of this paper carefully and providing many useful suggestions. The author would like to thank Guoliang Yu,  Emil Straube and Zhizhang Xie in Texas A \& M University, Kunyu Guo in Fudan University, and Xiang Tang in Washington University in St. Louis, for the valuable suggestions and many supports. She would also like to thank Harold Boas, Gregory Pearlstein, J. M. Landsberg in Texas A\& M University, and Mohan Ramachandran in SUNY Buffalo, for answering  many questions in several complex variables and algebraic geometry.
	
	\section{Preliminaries}\label{section: preliminary}
	
	This section contains some basic tools that we are going to use in this paper. Besides the classic tools in the study of operators on the Bergman space, we will also use the theory of complex analytic sets substantively. 
	
	\subsection{Arveson's Lemma}
	The following lemma provides an approach to the Arveson-Douglas Conjecture.
	\begin{lem}\label{lem: Arveson's lem} \cite{Arv05}
		Suppose $\mathcal{P}\subseteq\ber(\bn)$ is a submodule and $\mathcal{Q}$ is the corresponding quotient module. Then for any $p>n$, the following are equivalent.
		\begin{itemize}
			\item[(1)] $\mathcal{P}$ is $p$-essentially normal;
			\item[(2)] $\mathcal{Q}$ is $p$-essentially normal;
			\item[(3)] $[M_{z_i}, P]\in\mathcal{C}_{2p}$, $\forall i=1,\cdots,n$;
			\item[(4)] $[M_{z_i}, Q]\in\mathcal{C}_{2p}$, $\forall i=1,\cdots,n$.
		\end{itemize}
		Here $P, Q$ are the projections onto $\PP$ and $\QQ$, respectively.
	\end{lem}
	Notice that
	\[
	[M_{z_i}, Q]^*=QM_{z_i}^*-M_{z_i}^*Q=QM_{z_i}^*-QM_{z_i}^*Q=QM_{z_i}^*P.
	\]
	Compared with the cross commutators $[M_{z_i}, M_{z_j}^*]$, the operators $QM_{z_i}^*P$ are easier to work with. We will use Lemma \ref{lem: Arveson's lem} in the proofs in Section \ref{section ASD and EN}.
	
	\subsection{Complex Analytic Sets}
	The definitions and results come from \cite{Chi}. 
	\begin{defn}
		Let $\Omega$ be a complex manifold. A set $A\subset\Omega$ is called a \emph{(complex) analytic subset} of $\Omega$ if for each point $a\in\Omega$ there are a neighborhood $U$ of $a$ and functions $f_1,\cdots,f_N$ holomorphic in $U$ such that
		$$
		A\cap U=\{z\in U: f_1(z)=\cdots=f_N(z)=0\}.
		$$
		A point $a\in A$ is called \emph{regular} if there is a neighborhood $V$ of $a$ in $\Omega$ such that $A\cap V$ is a complex submanifold of $\Omega$. The (complex) dimension $\dim_a A$ at $a$ is naturally defined to be the (complex) dimension of the manifold $A\cap V$. 
		
		A point $a\in A$ is called a \emph{singular point} of $A$ if it is not regular. One can show that the set of regular points is dense in $A$. The dimension of $A$ at a singular point is defined as 
		\[
		\dim_a A=\limsup_{\substack{z\to a \\ z\in \mathrm{reg}A}}\dim_z A.
		\]
		$A$ is said to be of \emph{pure} dimension $p$ if $\dim_aA=p$, $\forall a\in A$.
	\end{defn}
	
	In this paper, our main objects of study are algebraic sets in $\cn$, i.e., zero loci of polynomials in $n$-variables.
	By Hilbert's Nullstellensatz, there is a one-to-one correspondence between algebraic sets and radical polynomial ideals, and irreducible algebraic sets correspond to prime ideals. We will also consider powers of radical ideals. Thus it is convenient to use the language of holomorphic chains.
	\begin{defn}\label{defn: holomorphic chain}
		A \emph{holomorphic chain} on a complex manifold $\Omega$ is a formal, locally finite sum $\sum k_j A_j$, where $A_j$ are pairwise distinct irreducible analytic subsets in $\Omega$ and $k_j\neq0$ are integers.  
	\end{defn}
	For $f\in\mathrm{Hol}(\Omega)$, we use the notation $f|_{\sum k_j A_j}=0$ to indicate that $f$ is, locally, a linear combination of functions of the form $\Pi_j\Pi_{i=1}^{k_j}f_{ij}$, where $f_{ij}$ are holomorphic functions vanishing on $A_j$.
	
	\begin{defn}\label{defn: proper maps, finite maps}
		A continuous map $f: X\to Y$ of topological spaces is called \emph{proper} if the pre-image of every compact set $K\subseteq Y$ is a compact set in $X$. The spaces $X$ and $Y$ are assumed to be Hausdorff and locally compact. 
	\end{defn}
	
	Proper maps are important tools in the study of analytic sets. The following results will be used in the proofs.
	\begin{thm}\label{thm: proper projections define analytic covers}
		Let $A$ be an analytic set in $\cn$, $a\in A$, $\dim_aA=p$, $0<p<n$, $U$ a neighborhood of $a$, and $\pi: A\cap U\to U'\subset\CC^p$, $z\mapsto z':=(z_1,\cdots,z_p)$ a proper projection. Then there is an analytic subset $\sigma\subset U'$ of dimension less than $p$ and a natural number $k$ such that
		\begin{itemize}
			\item[(1)] $\pi: A\cap U\backslash\pi^{-1}(\sigma)\to U'\backslash\sigma$ is a locally biholomorphic $k$-sheeted cover, in particular, $\#\pi^{-1}(z')\cap A\cap U=k$ for all $z'\in U\backslash\sigma$.
			\item[(2)] $\pi^{-1}(\sigma)$ is nowhere dense in $A_{(p)}\cap U$. Here $A_{(p)}=\{z\in A: \dim_zA=p\}$.
		\end{itemize}
		In particular, if $A$ is pure of dimension $p$, then $\pi^{-1}(\sigma)$ is nowhere dense in $A$. We say that $\pi$ defines a $k$-sheeted analytic cover. The set $\sigma$ is called the critical set of $\pi$.
	\end{thm}
	
	A proper projection on a complex analytic set $A$ gives rise to a set of canonical defining functions \cite{Chi}.
	\begin{defn}\label{defn: CDF}
		(1) Let $a_1,\cdots, a_k\subset\CC^m$, not necessarily distinct. We compose the polynomial
		\[
		P(z,w)=\la z-a_1, w\ra\cdots\la z-a_n, w\ra
		\]
		in the variable $(z,w)\in\CC^{2m}$. One can show that $P(z,w)\equiv^w 0$ if and only if $z$ is one of the points $a_1,\cdots,a_k$. Suppose
		\[
		P(z,w)=\sum_{|\alpha|=k}P_\alpha(z)\bar{w}^\alpha,
		\]
		where $\alpha$ denotes a multi-index $(\alpha_1,\cdots,\alpha_m)$. Then the condition $P(z,w)\equiv^w0$ is equivalent to that $P_\alpha(z)=0$, $\forall \alpha$, $|\alpha|=k$. The polynomials $P_\alpha(z)$ are called the \emph{canonical defining functions} for the system $a=\{a_1,\cdots,a_k\}$.
		
		\medskip
		(2) More generally, suppose $A$ is an analytic subset of $U=U'\times U''\subset\cn$, where $U'\subset\CC^p$, $U''\subset\CC^m$ and let $\pi: (z',z'')\mapsto z'\in\CC^p$. Suppose $\pi|_A:A\to U'$ is a $k$-sheeted analytic cover. Let $\sigma\subset\CC^p$ be the critical set of $\pi|_A$. For each $z'\in U'\backslash\sigma$, 
		\[
		\pi^{-1}(z')\cap A\cap U=\{(z', a_1(z')), \cdots, (z', a_k(z'))\},
		\]
		where $a_i(z')$ are holomorphic functions defined in a small neighborhood of $z'$. Define
		\[
		P_{\pi|_A}(z,w)=\la z''-a_1(z'), w\ra\cdots\la z''-a_k(z'), w\ra,~~~z\in U\backslash\pi^{-1}(\sigma), w\in\CC^m
		\]
		and
		\[
		P_{\pi|_A}(z,w)=\sum_{|\alpha|=k}P_{{\pi|_A},\alpha}(z)\overline{w^\alpha}.
		\]
		Here we write $\alpha=(\alpha_{p+1},\cdots,\alpha_n)$ to be consistant with the coordinates in $\cn$. The coefficients of powers of $z''$ in the functions $P_{{\pi|_A},\alpha}$ are locally bounded holomorphic functions on $U'\backslash\sigma$. Since $\dim\sigma<p$, they can be uniquely extended to holomorphic functions on $U'$. Therefore $P_{{\pi|_A},\alpha}$ extend to holomorphic functions on $U$ (in fact, on $U'\times\CC^m$). They are called the \emph{canonical defining functions} for the projection $\pi$.
		
		\medskip
		(3) One can also define canonical defining functions for holomorphic chains. We will only use the canonical defining functions for $mA$, where $m$ is a positive integer. Then we set $P_{\pi|_{mA}}(z,w)=P_\pi(z,w)^m$ and $P_{\pi|_{mA}}(z,w)=\sum_{|\alpha|=mk}P_{\pi|_{mA},\alpha}(z)\overline{w^\alpha}$. The functions $P_{\pi|_{mA},\alpha}(z)$ will be the canonical defining functions for the holomorphic chain $mA$.
	\end{defn}

\begin{rem}\label{rem: CDF basis}
	We remark that in Definition \ref{defn: CDF}, the functions $P_{\pi|_A}$ and $P_{\pi|_A,\alpha}$ are constructed under a specific choice of basis. Our estimates in this paper involve change of basis. It is convenient to generalize Definition \ref{defn: CDF} to the following ``coordinate-free'' form. Suppose $E=\{e_1,\cdots,e_n\}$ is an orthonormal basis of $\cn$ and $\pi$ is the orthonormal projection onto $\mathrm{span}\{e_1,\cdots,e_p\}$. Suppose $\pi|_A$ is proper. Define
	\[
	P_{\pi|_A, E}(z,w)=\Pi_{a\in\pi^{-1}\pi(z)\cap A}\la z-a,\sum_{i=p+1}^nw_ie_i\ra,\quad z\in\cn, w\in\CC^{n-p}.
	\]
	
	Suppose $l$ is a unitary transformation on $\cn$ and $\pi|_{l(A)}$ is also proper. Let $E_l=\{l^{-1}e_1,\cdots,l^{-1}e_n\}$. Then
	\[
	P_{\pi|_l(A), E}(l(z),w)=\Pi_{a\in\pi^{-1}\pi l(z)\cap l(A)}\la l(z)-a,\sum_{i=p+1}^nw_ie_i\ra=\Pi_{b\in l^{-1}\pi^{-1}\pi l(z)\cap A}\la z-b,\sum_{i=p+1}^nw_il^{-1}(e_i)\ra.
	\]
	In other words,
	\[
	P_{\pi|_l(A), E}(l(z),w)=P_{l^{-1}\pi l|_A, E_l}(z,w).
	\]
	We will use this fact in the proof of Lemma \ref{lem: decomposition on Fu}. In the subsequent discussions, we will omit the subscript $E$ where no confusion is caused.
\end{rem}

	\subsection{M\"{o}bius Transform and Bergman Metric}
	For $z\in\bn$, $z\neq0$, let $P_z$ and $Q_z$ be the orthogonal projections from $\cn$ to $\CC z$ and $z^\perp$, respectively. 
	\begin{defn}
		The \emph{M\"{o}bius transform} $\varphi_z$ is defined by the formula
		\[
		\varphi_z(w)=\frac{z-P_z(w)-(1-|z|^2)^{1/2}Q_z(w)}{1-\la w,z\ra},\quad w\in\bn.
		\]
	\end{defn}
	The following lemma contains some basic properties of the M\"{o}bius transform $\varphi_z$.
	One can find a proof in Chapter 2 of~\cite{Rudin}.
	
	\begin{lem}\label{basic about varphi}
		If $a$, $z$, $w\in\bn$, then
		\begin{itemize}
			\item[(1)] $$1-\langle\varphi_a(z),\varphi_a(w)\rangle=\frac{(1-\langle a,a\rangle)(1-\langle z,w\rangle)}{(1-\langle z,a\rangle)(1-\langle a,w\rangle)}.$$
			\item[(2)] As a consequence of (1),
			$$1-|\varphi_a(z)|^2=\frac{(1-|a|^2)(1-|z|^2)}{|1-\langle z,a\rangle|^2}.$$
			\item[(3)] The Jacobian of the automorphism~$\varphi_z$~is
			$$(J\varphi_z(w))=\frac{(1-|z|^2)^{n+1}}{|1-\langle w,z\rangle|^{2(n+1)}}.$$
		\end{itemize}
	\end{lem}
	
	\begin{defn}
		The \emph{pseudo-hyperbolic metric} $\rho$ is defined by
		\[
		\rho(z,w)=|\varphi_z(w)|,~~~z, w\in\bn.
		\]
		The \emph{hyperbolic metric} $\beta$ is defined by
		\[
		\beta(z,w)=\frac{1}{2}\log\frac{1+\rho(z,w)}{1-\rho(z,w)},~~~z, w\in\bn.
		\]
		$\beta$ is also called the \emph{Bergman metric} on $\bn$. For $r>0$ and $z\in\bn$, denote
		\[
		D(z,r)=\{w: \beta(z,w)<r\}.
		\]
	\end{defn}
	The two metrics $\rho$ and $\beta$ define the same topology on $\bn$. In the estimations, we will use whichever is more convenient. The following lemmas are straightforward to check. We omit the proofs.
	\begin{lem}\label{lem: beta rho}
		For $z, w\in\bn$, we have
		\begin{itemize}
			\item[(1)] $\beta(z,w)\in [-\frac{1}{2}\log(1-\rho^2(z,w)),\log2-\frac{1}{2}\log(1-\rho^2(z,w)))$.
			\item[(2)]$1-\rho^2(z,w)\in[e^{-2\beta(z,w)}, 4e^{-2\beta(z,w)})$.
		\end{itemize}
	\end{lem}
	
	\begin{lem}\label{basic computations}
		For $z, w\in\bn$, the following hold.
		\begin{itemize}
			\item[(1)]
			\[
			|1-\la z,w\ra|>\frac{1}{2}(1-|z|^2).
			\]	
			\item[(2)] 
			\[
			1-|\varphi_z(w)|^2\in(\frac{1}{4}(1-|z|^2)(1-|w|^2), 4\frac{1-|z|^2}{1-|w|^2}).
			\]
		\end{itemize}
	\end{lem}
	
	\subsection{Spherical Distance}
	The following definitions and lemmas will be used in Section \ref{section: a covering lemma}.
	\begin{defn}
		Let $S=\partial\bn$ be the unit sphere in $\cn$. For $\zeta, \xi\in S$, the \emph{spherical distance} $d(\zeta,\xi)$ is defined by
		\[
		d(\zeta,\xi)=|1-\la \zeta,\xi\ra|^{1/2}.
		\]
		Then $d$ defines a metric on $S$ (cf. \cite{Rudin}). For $\delta>0$, denote
		\[
		Q(\zeta,\delta)=\{\xi\in S: d(\xi,\zeta)<\delta\}.
		\]
	\end{defn}
	Let $\sigma$ denote the normalized surface measure on $S$, i.e., $\sigma(S)=1$. For $z, w\in\bn$, we will also write $d(z,w)=|1-\la z,w\ra|^{1/2}$. Then $d$ also satisfies the triangle inequality \cite[Proposition 5.1.2]{Rudin}.
	\begin{lem}{\cite[Proposition 5.1.4]{Rudin}}\label{Q delta size}
		When $n>1$, the ratio $\sigma(Q_\delta)/\delta^{2n}$ increases from $2^{-n}$ to a finite limit $A_0$ as $\delta$ decreases from $\sqrt{2}$ to $0$.
	\end{lem}
	
	On the punctured unit ball $\bn\backslash\{0\}$, consider the projection
	\[
	\pi_S:\bn\backslash\{0\}\to S,~~z\mapsto \frac{z}{|z|}.
	\]
	For $z, w\in\bn$, we will consider the spherical distance between their projections on $S$. Let us denote
	\[
	\ds(z,w)=d(\pi_S(z),\pi_S(w))=|1-\la\frac{z}{|z|},\frac{w}{|w|}\ra|^{1/2}.
	\]
	
	\begin{lem}\label{ds and d}
		For $z, w\in\bn$, we have
		\begin{itemize}
			\item[(1)] $d^2(z,w)<\ds^2(z,w)+(1-|z|^2)+(1-|w|^2)$.
			\item[(2)] $d^2(z,w)>\frac{1}{2}\ds^2(z,w)$.
		\end{itemize}
	\end{lem}
	
	\begin{proof}	
		By definition, 
		\[
		d^2(z,w)=|1-\la z,w\ra|,~~~~~\ds^2(z,w)=|1-\la \frac{z}{|z|},\frac{w}{|w|}\rangle|.
		\]
		Therefore,
		\[
		d^2(z,w)\leq (1-|z||w|)+|z||w||1-\la\frac{z}{|z|},\frac{w}{|w|}\ra|<(1-|z|^2)+(1-|w|^2)+\ds^2(z,w).
		\]
		This proves (1).
		
		The proof of (2) relies on the fact that $2|1-rc|>|1-c|$ for any $r\in(0,1)$ and $c\in\CC$, $|c|<1$. So
		\[
		d^2(z,w)=|1-\la z,w\ra|>\frac{1}{2}|1-\la\frac{z}{|z|},\frac{w}{|w|}\ra|=\frac{1}{2}\ds^2(z,w).
		\]
		This completes the proof.
	\end{proof}
	
	\subsection{An Inequality}
	In \cite{DGW}, the following theorem was proved, and then used to obtain $p$-essential normality of principal submodules.
	\begin{thm}\label{thm: key inequality}
		Suppose $h$ is a holomorphic function defined in a neighborhood of $\overline{\bn}$. Then there exist a constant $C>0$ and a positive integer $N$, such that for any $z, w\in\bn$ and any $f\in\mathrm{Hol}(\bn)$, we have
		\begin{equation}\label{eqn: key inequality}
		|h(z)f(w)|\leq C\frac{|1-\la z,w\ra|^N}{(1-|w|^2)^{n+1+N}}\int_{D(w,1)}|h(\lambda)f(\lambda)|dv(\lambda).
		\end{equation}
	\end{thm}
	
	The constants $C$ and $N$ depend on the function $h$. In the case when $h$ is a polynomial, the constants depend only on the degree of $h$. We provide a direct proof here.
	\begin{thm}\label{thm: key inequality for polynomial}
		Suppose $p$ is a polynomial and $N=\deg p$. Then for any $f\in\mathrm{Hol}(\bn)$ and $z, w\in\bn$,
		\begin{equation}\label{eqn: key inequality for polynomial}
		|p(z)f(w)|\leq C\frac{|1-\la z,w\ra|^N}{(1-|w|^2)^{n+1+N}}\int_{D(w,1)}|p(\lambda)f(\lambda)|dv(\lambda).
		\end{equation}
		The constant $C$ depends only on $N$.
	\end{thm}
	
	\begin{proof}
		For $w\in\bn$, $w\neq0$ and $a, b>0$, denote
		\[
		Q_w(a, b)=\{z\in\bn:~|P_w(z)-w|<a(1-|w|^2), |Q_w(z)|<b(1-|w|^2)^{1/2}\}.
		\]
		From \cite[2.2.7]{Rudin}, there exist $a, b$ such that $D(w, 1)$ contains $Q_w(a, b)$ for any $w\in\bn$, $w\neq0$. 
		
		For a polynomial $p$ with $\deg p=N$ and for $w\in\bn$, choose an orthonormal basis $\{e_1,\cdots, e_n\}$ such that $e_1=\frac{w}{|w|}$. Then $w=(|w|,0,\cdots,0)$. In the case $w=0$, choose any orthonormal basis. For any multi-index $\alpha=(\alpha_1,\cdots,\alpha_n)$ such that $|\alpha|\leq N$, applying \cite[Lemma 3.2]{DWK} to the one variable polynomial $\partial^{\alpha_2}\cdots\partial^{\alpha_n}p(\cdot,w_2,\cdots,w_n)$, we get
		\begin{eqnarray*}
			&&|\partial^\alpha p(|w|,0,\cdots,0)f(|w|,0,\cdots,0)|\\
			&\lesssim&\frac{1}{(1-|w|^2)^{\alpha_1+2}}\int_{|\lambda_1-|w||<a(1-|w|^2)}|\partial^{\alpha_2}\cdots\partial^{\alpha_n}p(\lambda_1,0,\cdots,0)f(\lambda_1,0,\cdots,0)|dv(\lambda_1).
		\end{eqnarray*}
		Applying  \cite[Lemma 3.2]{DWK} again to $\partial_3^{\alpha_3}\cdots\partial_n^{\alpha_n}p(\lambda_1,\cdot,w_3,\cdots,w_n)$, we get
		\begin{eqnarray*}
			&&|\partial_2^{\alpha_2}\cdots\partial_n^{\alpha_n}p(\lambda_1, 0,\cdots,0)f(\lambda_1,0,\cdots,0)|\\
			&\lesssim&\frac{1}{(1-|w|^2)^{\alpha_2/2+1}}\int_{|\lambda_2|<\frac{b}{\sqrt{n-1}}(1-|w|^2)^{1/2}}|\partial_3^{\alpha_3}\cdots\partial_n^{\alpha_n}p(\lambda_1,\lambda_2,0,\cdots,0)f(\lambda_1,\lambda_2,0,\cdots,0)|dv(\lambda_2).
		\end{eqnarray*}
		Inductively, for any $k=1,\cdots,n-1$,
		\begin{eqnarray*}
			&&|\partial_{k+1}^{\alpha_{k+1}}\cdots\partial_n^{\alpha_n}p(\lambda_1,\cdots,\lambda_k,0,\cdots,0)f(\lambda_1,\cdots,\lambda_k,0,\cdots,0)|\\
			&\lesssim&\frac{1}{(1-|w|^2)^{\alpha_k/2+1}}\int_{|\lambda_{k+1}|<\frac{b}{\sqrt{n-1}}(1-|w|^2)^{1/2}}|\partial_{k+2}^{\alpha_{k+2}}\cdots\partial_n^{\alpha_n}p(\lambda_1,\cdots,\lambda_{k+1},0,\cdots,0)\\
			&&\quad\quad f(\lambda_1,\cdots,\lambda_{k+1},0,\cdots,0)|dv(\lambda_{k+1}).
		\end{eqnarray*}
		Combining the inequalities above, we get
		\begin{eqnarray*}
			|\partial^\alpha p(w)f(w)|&\lesssim&\frac{1}{(1-|w|^2)^{\alpha_1+|\alpha'|/2+n+1}}\int_{\substack{|\lambda_1-w_1|<a(1-|w|^2)\\|\lambda'|<b(1-|w|^2)^{1/2}}}|p(\lambda)f(\lambda)|dv(\lambda)\\
			&\leq&\frac{1}{(1-|w|^2)^{\alpha_1+|\alpha'|/2+n+1}}\int_{D(w,1)}|p(\lambda)f(\lambda)|dv(\lambda).
		\end{eqnarray*}
		Since $p(z)=\sum_{|\alpha|\leq N}c_\alpha\partial^\alpha p(w)(z-w)^\alpha$, where $c_\alpha$ are the Taylor coefficients, we have
		\[
		|p(z)f(w)|\lesssim\sum_{|\alpha|\leq N}\frac{|z_1-|w||^{\alpha_1}|z_2|^{\alpha_2}\cdots|z_n|^{\alpha_n}}{(1-|w|^2)^{\alpha_1+|\alpha'|/2+n+1}}\int_{D(w,1)}|p(\lambda)f(\lambda)|dv(\lambda).
		\]
		Notice that
		\[
		\frac{|z_1-|w||^2}{|1-\la z,w\ra|^2}+\sum_{j=2}^n(1-|w|^2)\frac{|z_j|^2}{|1-\la z,w\ra|^2}=|\varphi_w(z)|^2<1.
		\]
		We have
		\[
		|z_1-|w||<|1-\la z,w\ra|,\quad |z_j|<\frac{|1-\la z,w\ra|}{(1-|w|^2)^{1/2}}, j=2,\cdots,n.
		\]
		Therefore
		\begin{eqnarray*}
			|p(z)f(w)|&\lesssim&\sum_{|\alpha|\leq N}\frac{|1-\la z,w\ra|^{\alpha_1+|\alpha'|/2}}{(1-|w|^2)^{\alpha_1+|\alpha'|/2+n+1}}\int_{D(w,1)}|p(\lambda)f(\lambda)|dv(\lambda)\\
			&\lesssim&\frac{|1-\la z,w\ra|^N}{(1-|w|^2)^{N+n+1}}\int_{D(w,1)}|p(\lambda)f(\lambda)|dv(\lambda).
		\end{eqnarray*}
		From the previous argument, we know that the controlling constant depends only on $N$. This completes the proof.
	\end{proof}
	
	\subsection{Some Useful Computations}
	The following inequality will be useful in subsequent estimates. Its proof is a direct application of the H\"{o}lder's inequality.
	\begin{lem}\label{sum and square}
		For a positive integer $M$ and $a_1, \cdots, a_M>0$, 
		\[
		(a_1+a_2+\cdots+a_M)^2\leq M(a_1^2+a_2^2+\cdots+a_M^2).
		\]
	\end{lem}
	
	We will use the following version of Schur's test.
	\begin{lem}\label{lem: Schur's test}
		Let $(X, d\mu)$ and $(X,d\nu)$ be measure spaces and $T$ be an integral operator with non-negative integral kernel $K(x,y)$,
		\[
		Tf(x)=\int_Xf(y)K(x,y)d\mu(y),\quad x,y\in X.
		\]
		Suppose there exist a $\mu$-measurable positive function $h$ and a $\nu$-measurable positive function $g$ on $X$ such that 
		\[
		\int_Xh(y)K(x,y)d\mu(y)\leq A g(x),\quad a.e. [\nu],
		\]
		and
		\[
		\int_Xg(x)K(x,y)d\nu(x)\leq B h(y),\quad a.e. [\mu].
		\]
		then $T$ defines a bounded operator from $L^2(X,d\mu)$ to $L^2(X,d\nu)$ and $\|T\|\leq A^{1/2}B^{1/2}$.
	\end{lem}
	We want to apply Schur's test to operators determined by the following integral kernels. For any $r>0$ and non-negative integers $l>0$, $0<d<n$, define
	\[
	T_lf(z)=\int_{\bn} f(w)\frac{(1-|w|^2)^l}{|1-\la z,w\ra|^{n+1+l}}dv(w),
	\]
	\[
	T_{d,l}f(z)=\int_{\bn} f(w)\frac{(1-|w'|^2)^l}{|1-\la z',w'\ra|^{n+1+l}}dv(w),
	\]
	\[
	\tilde{T}_lf(z)=\int_{\bn} f(w)\frac{(1-|w|^2)^l}{|1-\la z,w\ra|^{n+1/2+l}}dv(w),
	\]
	\[
	T_l^rf(z)=\int_{D(z,r)}f(w)\frac{(1-|w|^2)^l}{(1-\la z,w\ra)^{n+1+l}}dv(w),
	\]
	\[
	R_l^rf(z)=\int_{D(z,r)^c}f(w)\frac{(1-|w|^2)^l}{|1-\la z,w\ra|^{n+1+l}}dv(w)
	\]
	\[
	R_{l,d}^rf(z)=\int_{D((z',0),r)^c}f(w)\frac{(1-|w'|^2)^l}{|1-\la z',w'\ra|^{n+1+l}}dv(w).
	\]
	Here $w'=(w_1,\cdots,w_d)$.
	
	\begin{lem}\label{lem: norm estimates for integral kernels}
		For the operators defined above, the following hold.
		\begin{itemize}
			\item[(1)] For any positive integer $l$ and $0<d<n$, $T_l, T_{d,l}$ define bounded operators on $L^2(\bn)$. If $l\geq1$, then $T_l$ also defines a bounded operator on $L^2(\bn, (1-|z|^2)dv(z))$.
			\item[(2)] For $l\geq0$, $\tilde{T}_l$ defines a bounded operator from $L^2(\bn,(1-|z|^2)dv(z))$ to $L^2(\bn)$.
			\item[(3)] For any positive integer $l$ and for any $f\in\mathrm{Hol}(\bn)$,
			\[
			T_l^rf(z)=c_{r,l}f(z),
			\]
			where $c_{r,l}=\int_{D(0,r)}(1-|z|^2)^ldv(z)$.
			\item[(4)] For any $0<d<n$, any positive integer $l$ and any $r>0$, the operators $R_l^r$ and $R_{l,d}^r$ define bounded operators on $\ber(\bn)$. Moreover, 
			$$
			\max\{\|R_l^r\|, \|R_{l,d}^r\|\}\leq \epsilon_{r,l},
			$$
			where we have $\epsilon_{r,l}\to0$ as $r\to\infty$, for fixed $l$.
		\end{itemize}
	\end{lem}
	\begin{proof}
		We will only prove the statements for $T_{d,l}, \tilde{T}_l, T_l^r$ and $R_{l,d}^r$. We will use the Rudin-Forelli estimates \cite[Proposition 1.4.10]{Rudin}.
		
		Let $h(z)=(1-|z'|^2)^{-1/2}$. Then
		\begin{eqnarray*}
			&&\int_{\bn}\frac{(1-|w'|^2)^l}{|1-\la z',w'\ra|^{n+1+l}}h(w)dv(w)\\
			&=&\int_{w'\in\mathbb{B}_d}\frac{(1-|w'|^2)^{l-1/2}}{|1-\la z',w'\ra|^{n+1+l}}\int_{|w''|^2<1-|w'|^2}1dv_{n-d}(w'')dv_d(w')\\
			&\approx&\int_{\mathbb{B}_d}\frac{(1-|w'|^2)^{l+n-d-1/2}}{|1-\la z',w'\ra|^{n+1+l}}dv_d(w')\\
			&\lesssim&(1-|z'|^2)^{-1/2}=h(z).
		\end{eqnarray*}
		Here $v_k$ denotes the Lebesgue measure on $\CC^k$. Similarly, we have
		\[
		\int_{\bn}\frac{(1-|w'|^2)^l}{|1-\la z',w'\ra|^{n+1+l}}h(z)dv(z)\lesssim h(w).
		\]
		This proves (1).
		
		To prove (2), take $h(z)=(1-|z|^2)^{-1/2}$ and $g(w)=(1-|w|^2)^{-1}$. We omit the calculations.
		
		For any $f\in\mathrm{Hol}(\bn)$ and any $z\in\bn$,
		\begin{eqnarray*}
			T_l^rf(z)&=&\int_{D(z,r)}f(w)\frac{(1-|w|^2)^l}{(1-\la z,w\ra)^{n+1+l}}dv(w)\\
			&=&\int_{D(0,r)}f\circ\varphi_z(\lambda)\frac{(1-|\varphi_z(\lambda)|^2)^l}{(1-\la z,\varphi_z(\lambda)\ra)^{n+1+l}}\frac{(1-|z|^2)^{n+1}}{|1-\la z,\lambda\ra|^{2(n+1)}}dv(\lambda)\\
			&=&\int_{D(0,r)}f\circ\varphi_z(\lambda)\frac{(1-|\lambda|^2)^l}{(1-\la\lambda,z\ra)^{n+1+l}}dv(\lambda)\\
			&=&c_{r,l}f(z).
		\end{eqnarray*}
		This proves (3).
		
		Let $E=\{(z,w): \beta(w,(z',0))>r\}$. Denote $\beta_d$ the Bergman metric on $\mathbb{B}_d$. If $(z,w)\in E$, then $1-|\varphi_{(z',0)}(w)|^2<4e^{-2r}$. Since
		\begin{eqnarray*}
			1-|\varphi_{(z',0)}(w)|^2&=&\frac{(1-|z'|^2)(1-|w|^2)}{|1-\la z',w'\ra|^2}\\
			&=&(1-|\varphi_{z'}(w')|^2)\frac{1-|w|^2}{1-|w'|^2},
		\end{eqnarray*}
		either $1-|\varphi_{z'}(w')|^2<2e^{-r}$ or $\frac{1-|w|^2}{1-|w'|^2}<2e^{-r}$.
		Let $E_1=\{(z,w): \beta_d(z',w')>\frac{1}{2}r\}$ and $E_2=\{(z,w): \frac{1-|w|^2}{1-|w'|^2}<2e^{-r}\}$.
		Then $E\subset E_1\cup E_2$. It is easy to show that the integral kernels $\chi_{E_1}\frac{(1-|w'|^2)^l}{|1-\la z',w'\ra|^{n+1+l}}$ and $\chi_{E_2}\frac{(1-|w'|^2)^l}{|1-\la z',w'\ra|^{n+1+l}}$	define bounded operators with norms tending to $0$ as $r\to\infty$. This proves (4).
	\end{proof}
	
	We will also use the weighted Bergman norm. For $l$ a positive integer and $f\in\mathrm{Hol}(\bn)$,
	\[
	\|f\|_{L_{a,l}^2}^2=\int_{\bn}|f(z)|^2(1-|z|^2)^ldv(z).
	\]
	The following lemma is well known (cf. \cite{FX13}).
	\begin{lem}\label{lem: sob emb thm for weighted space}
		Let $T$ be a bounded linear operator on $\ber(\bn)$. If there exists a constant $C>0$ such that 
		\[
		\|Tf\|^2\leq C\|f\|_{L_{a,1}^2}^2,\quad\forall f\in\ber(\bn),
		\]
		then $T\in\mathcal{C}_p$ for all $p>2n$.
	\end{lem}

	\section{Asymptotic Stable Division Property and Essential Normality}\label{section ASD and EN}
	By Lemma \ref{lem: Arveson's lem}, in order to show that a submodule $\PP$ is $p$-essentially normal, $\forall p>n$, one needs to show that $QM_{z_i}^*P$ is in $\mathcal{C}_p$ for any $p>2n$. That means, for $f\in\PP$, one needs to find an element in $\PP$ that is close enough to $M_{z_i}^*f$. In the case when $\PP$ is principal with generator $h$, the set of functions $\{hf: f\in\poly\}$ is dense in $\PP$. For a function $hf$, a reasonable approximation of $M_{z_i}^*(hf)$ will be $hM_{z_i}^*f$ (cf. \cite{DGW}\cite{DWK}\cite{FX13}\cite{FX18}\cite{GWK}). In general, suppose $\PP$ is generated by $\{h_1,\cdots,h_k\}$, it may happen that $\sum_{j=1}^kh_jf_j$ equals $0$  while $\sum_{j=1}^kh_jM_{z_i}^*f_j$ does not. Thus the distance between $M_{z_i}^*(\sum_{j=1}^kh_jf_j)$ and $\sum_{j=1}^kh_jM_{z_i}^*f_j$ may not be small (compared to $\|\sum_{j=1}^kh_jf_j\|_{L_{a,1}^2}$). 
	
	One can avoid such problems by putting restrictions on the decomposition of $f$. In \cite{Sha}, Shalit considered submodules of the Drury-Arveson module, with the stable division property. For a submodule with the stable division property, one can always find a decomposition $f=\sum_{j=1}^kh_jf_j$ with 
	\begin{equation}\label{eqn: SD}
	\sum_{j=1}^k\|h_jf_j\|\leq C\|f\|,
	\end{equation}
	where $C$ is a constant depending only on $\PP$. Shalit showed that graded submodules with stable division property are essentially normal.
	
	We propose the following definition of asymptotic stable division property.
	\begin{defn}\label{defn: ASD}
		Suppose $\PP$ is a submodule of the Bergman module $\ber(\bn)$. $\PP$ is said to have the \emph{asymptotic stable division property} if there exist an invertible operator $T$ on $\PP$, a subset $\{h_i\}_{i\in\Lambda}\subset\PP$, finite or countably infinite, and constants $C_1, C_2$, such that for any $f\in\PP$, there exists $\{g_i\}_{i\in\Lambda}\subset\mathrm{Hol}(\bn)$ with the following properties.
		\begin{itemize}
			\item[(1)] $Tf=\sum_{i\in\Lambda}h_ig_i$, where the convergence is pointwise if $\Lambda$ is countably infinite.
			\item[(2)] 
			\[
			\int_{\bn}\bigg(\sum_{i\in\Lambda}|h_i(z)g_i(z)|\bigg)^2dv(z)\leq C_1\|f\|^2_{\ber}.
			\]
			\item[(3)] 
			\[
			\int_{\bn}\bigg(\sum_{i\in\Lambda}|h_i(z)g_i(z)|\bigg)^2(1-|z|^2)dv(z)\leq C_2\|f\|_{L_{a,1}^2}^2.
			\]
		\end{itemize}
	\end{defn}
	
	Similar to the case of stable division property, we have the following theorem.
	\begin{thm}\label{thm: ASD to EN}
		Suppose $\PP$ is a submodule of $\ber(\bn)$ with the asymptotic stable division property. If the generating functions $h_i$ are all defined in neighborhoods of $\overline{\bn}$ and the sets of constants $\{C_i\}_{i\in\Lambda}$, $\{N_i\}_{i\in\Lambda}$, determined by $h_i$ (as in Theorem \ref{thm: key inequality}), are bounded, then the submodule $\PP$ is $p$-essentially normal for all $p>n$. In particular, if the generating functions $h_i$ are polynomials of bounded degrees, then $\PP$ is $p$-essentially normal for all $p>n$.
	\end{thm}
	
	\begin{proof}
		Denote $P$ the projection operator onto $\PP$ and $Q$ the projection operator onto $\PP^\perp$.
		By Lemma \ref{lem: Arveson's lem}, it suffices to show that $[M_{z_k}^*, P]=QM_{z_k}^*P$ is in $\mathcal{C}_p$, $\forall p>2n$. Let $N=\max\{N_i: i\in\Lambda\}$. Since $C_i$ are uniformly bounded, by Lemma \ref{basic computations} (1), there is a constant $C$ such that inequality \eqref{eqn: key inequality} holds for all $h_i$ with constants $C$ and $N$. Choose a positive integer $l>N$. Define
		\[
		M_{z_k}^{(l)*}f(z)=c_l^{-1}\int_{\bn}\bar{w_k}f(w)K_w^{(l)}(z)(1-|w|^2)^ldv(w),
		\]
		where $K_w^{(l)}(z)=\frac{1}{(1-\la z,w\ra)^{n+l+1}}$ is the weighted reproducing kernel, and $c_l=\int_{\bn}(1-|w|^2)^ldv(w)$.
		For $f\in\ber(\bn)$, 
		\begin{eqnarray*}
			&&|M_{z_k}^*f(z)-M_{z_k}^{(l)*}f(z)|\\
			&=&\bigg|\int_{\bn}(\bar{w_k}-\bar{z_k})f(w)K_w(z)dv(w)-c_l^{-1}\int_{\bn}(\bar{w_k}-\bar{z_k})f(w)K_w^{(l)}(z)(1-|w|^2)^ldv(w)\bigg|\\
			&\lesssim&\int_{\bn}|f(w)|\frac{|w-z|}{|1-\la z,w\ra|^{n+1}}dv(w)\\
			&\lesssim&\int_{\bn}|f(w)|\frac{1}{|1-\la z,w\ra|^{n+1/2}}dv(w).
		\end{eqnarray*}
		By Lemma \ref{lem: norm estimates for integral kernels} and Lemma \ref{lem: sob emb thm for weighted space}, $M_{z_k}^*-M_{z_k}^{(l)*}$ is in $\mathcal{C}_p$ for any $p>2n$.
		
		For $f\in\PP$, by assumption, $Tf=\sum_{i\in\Lambda}h_ig_i$.
		Define
		\[
		S_kf(z)=\sum_{i\in\Lambda}h_i(z)G_i(z),
		\]
		where
		\[
		G_i(z)=c_l^{-1}\int \bar{w_k}g_i(w)K_w^{(l)}(z)(1-|w|^2)^ldv(w),\quad i\in\Lambda.
		\]
		For each $i\in\Lambda$, applying Theorem \ref{thm: key inequality}, we get
		\begin{eqnarray*}
			|h_i(z)G_i(z)|&=&c_l^{-1}\bigg|h_i(z)\int \bar{w_k}g_i(w)K_w^{(l)}(z)(1-|w|^2)^ldv(w)\bigg|\\
			&\leq&Cc_l^{-1}\int_{\bn}\frac{|1-\la z,w\ra|^N}{(1-|w|^2)^{n+N+1}}\int_{D(w,1)}|h_i(\lambda)g_i(\lambda)|dv(\lambda)\frac{(1-|w|^2)^l}{|1-\la z,w\ra|^{n+l+1}}dv(w)\\
			&=&Cc_l^{-1}\int_{\bn}|h_i(\lambda)g_i(\lambda)|\int_{D(\lambda,1)}\frac{1}{(1-|w|^2)^{n+1+N-l}|1-\la z,w\ra|^{n+1+l-N}}dv(w)dv(\lambda)\\
			&\lesssim&C\int_{\bn}|h_i(\lambda)g_i(\lambda)|\frac{(1-|\lambda|^2)^{l-N}}{|1-\la z,\lambda\ra|^{n+1+l-N}}dv(\lambda).
		\end{eqnarray*}
		By Lemma \ref{lem: norm estimates for integral kernels} and our assumption, $h_iG_i$ belongs to $\ber(\bn)$. By \cite[Proposition 5.5]{DGW}, $h_iG_i$ is in the principal submodule generated by $h_i$, which is contained in $\PP$. Moreover,
		\[
		\sum_{i\in\Lambda}|h_i(z)G_i(z)|\lesssim C\int_{\bn}\bigg(\sum_{i\in\Lambda}|h_i(\lambda)g_i(\lambda)|\bigg)\frac{(1-|\lambda|^2)^{l-N}}{|1-\la z,\lambda\ra|^{n+1+l-N}}dv(\lambda).
		\]
		By condition (2) in Definition \ref{defn: ASD}, the series $\sum_{i\in\Lambda}h_iG_i$ converges weakly. Therefore $S_kf=\sum_{i\in\Lambda}h_iG_i\in\PP$.
		
		Next, we show that $QM_{z_k}^{(l)*}P$ is in $\mathcal{C}_p$ for any $p>2n$. Once this is done, we will have $QM_{z_k}^*P=Q(M_{z_k}^*-M_{z_k}^{(l)*})P+QM_{z_k}^{(l)*}P\in\mathcal{C}_p$, which is exactly what we need.
		
		For any $f\in\PP$,
		\begin{eqnarray*}
			&&|M_{z_k}^{(l)*}Tf(z)-S_kf(z)|\\
			&\leq&c_l^{-1}\sum_{i\in\Lambda}\bigg|\int\bar{w_k}h_i(w)g_i(w)K_w^{(l)}(z)(1-|w|^2)^ldv(w)\\
			&&-\int\bar{w_k}h_i(z)g_i(w)K_w^{(l)}(z)(1-|w|^2)^ldv(w)\bigg|\\
			&=&c_l^{-1}\sum_{i\in\Lambda}\bigg|\int(\bar{w_k}-\bar{z_k})h_i(w)g_i(w)K_w^{(l)}(z)(1-|w|^2)^ldv(w)\\
			&&-\int(\bar{w_k}-\bar{z_k})h_i(z)g_i(w)K_w^{(l)}(z)(1-|w|^2)^ldv(w)\bigg|\\
			&\lesssim&\sum_{i\in\Lambda}\int|h_i(w)g_i(w)|\frac{(1-|w|^2)^l}{|1-\la z,w\ra|^{n+1/2+l}}dv(w)\\
			&&+\sum_{i\in\Lambda}\int|h_i(z)g_i(w)|\frac{(1-|w|^2)^l}{|1-\la z,w\ra|^{n+1/2+l}}dv(w).
		\end{eqnarray*}
		Again,
		\begin{eqnarray*}
			&&\int|h_i(z)g_i(w)|\frac{(1-|w|^2)^l}{|1-\la z,w\ra|^{n+1/2+l}}dv(w)\\
			&\lesssim&\int\frac{|1-\la z,w\ra|^N}{(1-|w|^2)^{n+1+N}}\int_{D(w,1)}|h_i(\lambda)g_i(\lambda)|dv(\lambda)\frac{(1-|w|^2)^l}{|1-\la z,w\ra|^{n+1/2+l}}dv(w)\\
			&\lesssim&\int|h_i(\lambda)g_i(\lambda)|\frac{(1-|\lambda|^2)^{l-N}}{|1-\la z,\lambda\ra|^{n+1/2+l-N}}dv(\lambda).
		\end{eqnarray*}
		Thus
		\[
		|M_{z_k}^{(l)*}Tf(z)-S_kf(z)|\lesssim\int\bigg(\sum_{i\in\Lambda}|h_i(w)g_i(w)|\bigg)\frac{(1-|w|^2)^{l-N}}{|1-\la z,w\ra|^{n+1/2+l-N}}dv(w).
		\]
		Applying Lemma \ref{lem: norm estimates for integral kernels} (2) on the right-hand side, we obtain
		\[
		\|QM_{z_k}^{(l)*}Tf\|^2\leq\|M_{z_k}^{(l)*}Tf-S_k f\|^2\lesssim\int\bigg(\sum_{i\in\Lambda}|h_i(w)g_i(w)|\bigg)^2(1-|w|^2)dv(w)\lesssim\|f\|_{L_{a,1}^2}^2.
		\]
		By Lemma \ref{lem: sob emb thm for weighted space}, $QM_{z_k}^{(l)*}TP$ is in $\mathcal{C}_p$ for any $p>2n$. Since $T$ is invertible on $\PP$, we have $QM_{z_k}^{(l)*}P=QM_{z_k}^{(l)*}TPT^{-1}P\in\mathcal{C}_p$, $\forall p>2n$.
		  Thus $QM_{z_k}^*P$ is in $\mathcal{C}_p$ for any $p>2n$. By Lemma \ref{lem: Arveson's lem}, $\PP$ is $p$-essentially normal for all $p>n$. This completes the proof.
	\end{proof}

	\begin{rem}\label{rem: examples of ASD}
		As indicated in the title, the aim of this paper is to find a unified proof that works for most known results of the Arveson-Douglas Conjecture. First, suppose $h$ is a holomorphic function defined in a neighborhood of $\overline{\bn}$ and $\PP_h$ is the principal submodule generated by $h$. Then $\PP_h$ has the asymptotic stable division property trivially. Second, Theorem \ref{thm: ASD to EN} generalizes Shalit's result in that we do not require the generators to be polynomials and we do not require the submodule to be graded. In fact, by Theorem \ref{thm: key inequality} and Theorem \ref{thm: ASD to EN}, any finite set of generators $\{h_i\}_{i=1}^k$, defined in a neighborhood of $\overline{\bn}$, satisfying inequality \eqref{eqn: SD} for relevant norms, generate an essentially normal submodule. Finally, we will show in Theorem \ref{thm: Smooth implies ASD}
		that most of the submodules in \cite{DTY}\cite{DWY1}\cite{EE} have the asymptotic stable division property. 
	\end{rem}
	
	Before proving Theorem \ref{thm: Smooth implies ASD},
	let us discuss the matter with some generality. Consider the following technical hypotheses.
	
	\medskip
	
	\noindent\textbf{Hypothesis 1}: Suppose $I\subset\poly$ is an ideal, $N$, $M$ are positive integers, $C>0$. For any $\epsilon>0$ sufficiently small, there exists $R_0>0$ with the following property. For $R>R_0$, there exist constants $0<\delta<1$, $C'>0$, open covers $\{E_i\}_{i\in\Lambda}$, $\{F_i\}_{i\in\Lambda}$, finite or countably infinite, of $\Bd:=\{z\in\bn:~|z|>\delta\}$, and for each $i\in\Lambda$, a subset $\{\pij\}_{j\in\Gamma_i}\subseteq I$, finite or countably infinite, such that the following hold.
	\begin{itemize}
		\item[(1)] $E_i\subset F_i\subset\bn$.
		\item[(2)] For $k\in\mathbb{N}$, denote
		\[
		E_{ik}=\{w\in\bn: \beta(w, E_i)<kR\}.
		\]
		Then $E_{i3}\subset F_i$. Moreover, any $z\in\bn$ belongs to at most $M$ of the sets $\{F_i\}_{i\in\Lambda}$.
		\item[(3)] For any $i\in\Lambda, j\in\Gamma_i$, $\deg\pij<N$.
		\item[(4)] For any $i\in\Lambda$ and any $f\in \PP_I$, there exist $\{\fij\}_{j\in\Gamma_i}\subseteq\mathrm{Hol}(E_{i3})$ such that 
		\begin{itemize}
			\item[(i)] 
			\[
			\int_{E_{i3}}|\sum_{j\in\Gamma_i} \pij(\lambda)\fij(\lambda)-f(\lambda)|^2dv(\lambda)\leq\epsilon\int_{F_i}|f(\lambda)|^2dv(\lambda).
			\]
			\item[(ii)]
			\[
			\int_{E_{i3}}\bigg(\sum_{j\in\Gamma_i} |\pij(\lambda)\fij(\lambda)|\bigg)^2dv(\lambda)\leq C\int_{F_i}|f(\lambda)|^2dv(\lambda).
			\]
			\item[(iii)]
			\begin{eqnarray*}
				&&\int_{E_{i3}}\bigg(\sum_{j\in\Gamma_i} |\pij(\lambda)\fij(\lambda)|\bigg)^2(1-|\lambda|^2)dv(\lambda)\\
				&\leq&C'\int_{F_i}|f(\lambda)|^2(1-|\lambda|^2)dv(\lambda).
			\end{eqnarray*}
			\item[(iv)] The maps $f\mapsto \fij$ are linear.
		\end{itemize}
	\end{itemize}
	
	\begin{thm}\label{thm: ASD}
		Under Hypothesis 1, $\PP_I$ has the asymptotic stable division property with generating functions consisting of polynomials of bounded degrees.
	\end{thm}
	
	Combining Theorem  \ref{thm: ASD to EN} and Theorem \ref{thm: ASD}, we have the following corollary.
	\begin{coro}\label{cor: hypo1 to EN}
		Assume Hypothesis 1. Then the submodule $\PP_I$ is $p$-essentially normal for all $p>n$.
	\end{coro}
	
	\begin{proof}[\textbf{Proof of Theorem \ref{thm: ASD}}]
		Let $I$, $N, M, C$ be as in Hypothesis 1. 
		Let $\epsilon>0$, $R>r>0$ be determined later. We will always assume that $R>R_0$, where $R_0$ is determined by $\epsilon$, as in Hypothesis 1. For $i\in\Lambda$, define
		\[
		\varphi_i(z)=\begin{cases}
		1-\frac{\beta(z, E_i)}{R}, &\beta(z, E_i)\leq R\\
		0, &\text{ otherwise.}
		\end{cases}
		\]
		It is easy to see that 
		\begin{itemize}
			\item[(1)]$0\leq\varphi_i\leq 1$.
			\item[(2)]$\{z: \varphi_i(z)\neq0\}= E_{i1}$.
			\item[(3)]$|\varphi_i(z)-\varphi_i(w)|\leq\frac{\beta(z,w)}{R}$.
		\end{itemize}
		
		Fix a positive integer $l>N$. Define the linear operator
		\[
		\T: I\to\PP_I,~~~f\mapsto \T f(z)=\sum_{i\in\Lambda,j\in\Gamma_i}\pij(z)\gij(z),
		\]
		where 
		\[
		\gij(z)=\int_{\bn}\varphi_i(w)\fij(w)K_w^{(l)}(z)(1-|w|^2)^ldv(w).
		\]
		First, we show that $\T$ is well-defined and extends to a bounded operator on $\PP_I$. For each pair $i, j$ and any $z\in\bn$,
		\begin{eqnarray}\label{ineq: pij(z)gij(z)}
		&&|\pij(z)\gij(z)|\nonumber\\
		&\leq&|\pij(z)|\int\varphi_i(w)|\fij(w)|\frac{(1-|w|^2)^l}{|1-\la z,w\ra|^{n+1+l}}dv(w) \nonumber\\
		&\lesssim&\int_{E_{i1}}\frac{|1-\la z,w\ra|^N}{(1-|w|^2)^{n+1+N}}\int_{D(w,1)}|\pij(\lambda)\fij(\lambda)|dv(\lambda)\frac{(1-|w|^2)^l}{|1-\la z,w\ra|^{n+1+l}}dv(w)\nonumber\\
		&\lesssim&\int_{E_{i2}}|\pij(\lambda)\fij(\lambda)|\frac{(1-|\lambda|^2)^{l-N}}{|1-\la z,\lambda\ra|^{n+1+l-N}}dv(\lambda).
		\end{eqnarray}
		Here the second inequality comes from Theorem \ref{thm: key inequality for polynomial}.
		Hence
		\[
		\|\pij\gij\|^2\lesssim \int_{E_{i2}}|\pij\fij|^2dv\leq C\|f\|^2.
		\]
		Therefore the map $f\to\pij\gij$ extends to a bounded linear operator from $\PP_I$ to the principal submodule generated by $\pij$. By \cite[Proposition 5.5]{DGW}, the image for any $f\in\PP_I$ equals $\pij\gij$ for some $\gij\in\mathrm{Hol}(\bn)$.
		
		Next, we show that $\T$ itself extends to a bounded linear operator. From inequality (\ref{ineq: pij(z)gij(z)}), we see that
		\begin{eqnarray*}
			\sum_{i,j}|\pij(z)\gij(z)|&\lesssim&\sum_{i,j}\int_{E_{i2}}|\pij(\lambda)\fij(\lambda)|\frac{(1-|\lambda|^2)^{l-N}}{|1-\la z,\lambda\ra|^{n+1+l-N}}dv(\lambda)\\
			&=&\int_{\bn}\bigg(\sum_i\chi_{E_{i2}}(\lambda)\sum_j|\pij(\lambda)\fij(\lambda)|\bigg)\frac{(1-|\lambda|^2)^{l-N}}{|1-\la z,\lambda\ra|^{n+1+l-N}}dv(\lambda).
		\end{eqnarray*}
		By Lemma \ref{lem: norm estimates for integral kernels} (1),
		\begin{eqnarray*}
			\|\T f\|^2&\leq&\|\sum_{i,j}|\pij\gij|\|^2\\
			&\lesssim&\int_{\bn}\bigg(\sum_i\chi_{E_{i2}}(\lambda)\sum_j|\pij(\lambda)\fij(\lambda)|\bigg)^2dv(\lambda).
		\end{eqnarray*}
		By our hypotheses, for each $\lambda$, there are at most $M$ functions $\chi_{E_{i2}}$ with $\chi_{E_{i2}}(\lambda)\neq0$. By Lemma \ref{sum and square},
		\[
		\bigg(\sum_i\chi_{E_{i2}}(\lambda)\sum_j|\pij(\lambda)\fij(\lambda)|\bigg)^2\leq M\sum_i\chi_{E_{i2}}(\lambda)\bigg(\sum_j|\pij(\lambda)\fij(\lambda)|\bigg)^2.
		\]
		Therefore
		\begin{eqnarray*}
			\|\T f\|^2\leq\|\sum_{i,j}|\pij\gij|\|^2&\lesssim&M\sum_i\int_{E_{i2}}\bigg(\sum_j|\pij(\lambda)\fij(\lambda)|\bigg)^2dv(\lambda)\\
			&\leq&CM\sum_i\int_{F_i}|f(\lambda)|^2dv(\lambda)\\
			&\leq&CM^2\|f\|^2.
		\end{eqnarray*}
		Hence $\T$ defines a bounded operator. 
		
		Arguments similar to those in the proof of Theorem \ref{thm: ASD to EN} show that $\T f\in\PP_I$ for all $f\in I$.
		
		To sum it up, $\T$ extends to a bounded linear operator on $\PP_I$ and for each $f\in\PP_I$, $\T f$ has the form $\sum_{i, j}\pij\gij$. Moreover, 
		\begin{equation}\label{eqn: 3.4.1}
		\|\sum_{i,j}|\pij\gij|\|^2\lesssim CM^2\|f\|^2.
		\end{equation}
		Since $l>N$, by Lemma \ref{lem: norm estimates for integral kernels}, the estimates above will also give us the following:
		\begin{equation}\label{eqn: 3.4.2}
		\int\bigg(\sum_{i,j}|\pij(z)\gij(z)|\bigg)^2(1-|z|^2)dv(z)\lesssim C'M^2\|f\|^2_{L_{a,1}^2}.
		\end{equation}
		
		Next, we want to show that inequality (1) in Definition \ref{defn: ASD} holds. We will obtain this by showing that a finite rank perturbation of $\T$ is invertible on $\PP_I$. For any $z\in\bn$ and the chosen $r<R$,
		\begin{eqnarray*}
			&&\big|\T f(z)-\sum_{i,j}\pij(z)\int_{D(z,r)}\varphi_i(w)\fij(w)K_w^{(l)}(z)(1-|w|^2)^ldv(w)\big|\\
			&\leq&\sum_{i, j}|\pij(z)|\int_{E_{i1}\backslash D(z,r)}|\fij(w)|\frac{(1-|w|^2)^l}{|1-\la z,w\ra|^{n+1+l}}dv(w)\\
			&\lesssim&\sum_{i, j}\int_{E_{i1}\backslash D(z,r)}\frac{|1-\la z,w\ra|^N}{(1-|w|^2)^{n+1+N}}\int_{D(w,1)}|\pij(\lambda)\fij(\lambda)|dv(\lambda)\frac{(1-|w|^2)^l}{|1-\la z,w\ra|^{n+1+l}}dv(w)\\
			&\lesssim&\sum_{i, j}\int_{E_{i2}\backslash D(z,r-1)}|\pij(\lambda)\fij(\lambda)|\frac{(1-|\lambda|^2)^{l-N}}{|1-\la z,\lambda\ra|^{n+1+l-N}}dv(\lambda)\\
			&=&\int_{D(z,r-1)^c}\bigg(\sum_{i,j}\chi_{E_{i2}}(\lambda)|\pij(\lambda)\fij(\lambda)|\bigg)\frac{(1-|\lambda|^2)^{l-N}}{|1-\la z,\lambda\ra|^{n+1+l-N}}dv(\lambda).
		\end{eqnarray*}
		
		Also,
		\begin{eqnarray*}
			&&\big|\sum_{i, j}\pij(z)\int_{D(z,r)}\varphi_i(w)\fij(w)K_w^{(l)}(z)(1-|w|^2)^ldv(w)\\
			&&-\sum_{i, j}\pij(z)\varphi_i(z)\int_{D(z,r)}\fij(w)K_w^{(l)}(z)(1-|w|^2)^ldv(w)\big|\\
			&\leq&\sum_{i, j}|\pij(z)|\int_{E_{i2}}\frac{r}{R}|\fij(w)|\frac{(1-|w|^2)^l}{|1-\la z,w\ra|^{n+1+l}}dv(w)\\
			&\lesssim&\frac{r}{R}\sum_{i, j}\int_{E_{i2}}\frac{|1-\la z,w\ra|^N}{(1-|w|^2)^{n+1+N}}\int_{D(w,1)}|\pij(\lambda)\fij(\lambda)|dv(\lambda)\frac{(1-|w|^2)^l}{|1-\la z,w\ra|^{n+1+l}}dv(w)\\
			&\lesssim&\frac{r}{R}\sum_{i, j}\int_{E_{i3}}|\pij(\lambda)\fij(\lambda)|\frac{(1-|\lambda|^2)^{l-N}}{|1-\la z,\lambda\ra|^{n+1+l-N}}dv(\lambda)\\
			&=&\frac{r}{R}\int\bigg(\sum_{i, j}\chi_{E_{i3}}(\lambda)|\pij(\lambda)\fij(\lambda)|\bigg)\frac{(1-|\lambda|^2)^{l-N}}{|1-\la z,\lambda\ra|^{n+1+l-N}}dv(\lambda).
		\end{eqnarray*}
		Finally, notice  that if $\varphi_i(z)\neq0$, then $z\in E_{i1}$. By Lemma \ref{lem: norm estimates for integral kernels} (3),
		\[
		\sum_{i, j}\pij(z)\varphi_i(z)\int_{D(z,r)}\fij(w)K_w^{(l)}(z)(1-|w|^2)^ldv(w)=c_{r,l}\sum_{i, j}\varphi_i(z)\pij(z)\fij(z).
		\]
		By Lemma \ref{lem: norm estimates for integral kernels} and the previous arguments,
		\[
		\|\T f-c_{r,l}\sum_i\varphi_i\sum_j\pij\fij\|\lesssim (\epsilon_{r-1,l-N}+\frac{r}{R})\|\sum_{i, j}\chi_{E_{i3}}|\pij\fij|\|\lesssim(\epsilon_{r-1,l-N}+\frac{r}{R})(CM)^{1/2}\|f\|.
		\]
		
		On the other hand, by Hypothesis 1,
		\[
		\int_{E_{i3}}|\sum_j\pij(\lambda)\fij(\lambda)-f(\lambda)|^2dv(\lambda)\leq\epsilon\int_{F_i}|f(\lambda)|^2dv(\lambda).
		\]
		So
		\begin{eqnarray*}
			&&\int_{\bn}|\sum_i\varphi_i(\lambda)\sum_j\pij(\lambda)\fij(\lambda)-\sum_i\varphi_i(\lambda)f(\lambda)|^2dv(\lambda)\\
			&\leq&M\sum_i\int_{E_{i1}}|\sum_j\pij(\lambda)\fij(\lambda)-f(\lambda)|^2dv(\lambda)\\
			&\leq&\epsilon M\sum_i\int_{F_i}|f(\lambda)|^2dv(\lambda)\\
			&\leq&\epsilon M^2\|f\|^2.
		\end{eqnarray*}
		Combining the above estimates, we get
		\[
		\|\T f-c_{r,l}\sum_i\varphi_i f\|\lesssim \big[(\epsilon_{r-1,l-N}+\frac{r}{R})(CM)^{1/2}+\epsilon^{1/2}M\big]\|f\|.
		\]
		Let $T_{\sum_i\varphi_i}$ be the Toeplitz operator with symbol $\sum_i\varphi_i$. Then the above implies
		\begin{equation}\label{ineqn: Tf -sum varphi_i f}
		\|\T f-c_{r,l}T_{\sum_i\varphi_i} f\|\lesssim \big[(\epsilon_{r-1,l-N}+\frac{r}{R})(CM)^{1/2}+\epsilon^{1/2}M\big]\|f\|.
		\end{equation}
		
		Since $C, l, N$ and $M$ are fixed, we can choose $\epsilon$, $r$, and then $R>R_0$ so that 
		\[
		(\epsilon_{r-1,l-N}+\frac{r}{R})(CM)^{1/2}+\epsilon^{1/2}M<1/2c_{r,l}.
		\]
		
		For a positive integer $k$, denote 
		\[
		I_k=\{q\in I, \deg q\leq k\}.
		\]
		Since $\sum\varphi_i\geq\chi_{\mathbb{B}_\delta}$,
		we can choose $k$ large enough so that the compression of $T_{\sum_i\varphi_i}$ on $\PP_I\ominus I_k$ is invertible on $\PP_I\ominus I_k$ and the norm of its inverse is less than 2. Then the compression of $\T$ on $\PP_I\ominus I_k$ is invertible. By a block matrix argument, it is easy to see that $T:=\T(P-P_{I_k})+P_{I_k}$ is invertible on $\PP_I$. Here $P$ denotes the orthogonal projection onto $\PP_I$.   Choose an orthonormal basis $\{h_i\}$ for $I_k$ under the Bergman norm. For any $f\in\PP_I$,
		\[
		Tf=\sum_{i,j}\pij\gij+P_{I_k}f=\sum_{i,j}\pij\gij+\sum_ic_ih_i.
		\]
		Then by inequalities \eqref{eqn: 3.4.1} and \eqref{eqn: 3.4.2},
		\[
		\int\big(\sum_{i,j}|\pij\gij|\big)^2dv\lesssim\|(P-P_{I_k})f\|^2\leq \|f\|^2.
		\]
		
		Since $I_k$ is finitely dimensional, we also have
		\[
		\int\big(\sum_i|c_ih_i|\big)^2dv\lesssim \|\sum_ic_ih_i\|^2\lesssim\|f\|^2.
		\]
		
		It is also easy to see that $P_{I_k}$ extends to a bounded operator on the closure of $I$ in $L_{a,1}^2(\bn)$. Therefore 
		\[
		\int\big(\sum_i|c_ih_i|\big)^2(1-|z|^2)dv\lesssim \|f\|^2_{L_{a,1}^2}.
		\]
		
		and
		\[
		\int\big(\sum_{i,j}|\pij\gij|\big)^2(1-|z|^2)dv\lesssim\|f-P_{I_k}f\|_{L_{a,1}^2}^2\lesssim  \|f\|^2_{L_{a,1}^2}.
		\]
		 This completes the proof.
	\end{proof}

	\section{A Distance Estimate}\label{section: a distance estimate}

	As promised in Remark \ref{rem: examples of ASD}, we are going to show that most submodules in the known results of the Geometric Arveson-Douglas Conjecture has the asymptotic stable division property. In fact, we will prove the following more general result.
	\begin{thm}\label{thm: Smooth implies ASD}
		Suppose $I$ is an ideal in $\poly$ with primary decomposition $I=\cap_{j=1}^k I_j^{m_j}$, where $I_j$ are prime ideals. Assume the following.
		\begin{itemize}
			\item[(1)] For any $j=1,\cdots,k$, $Z(I_j)$ has no singular points on $\partial\bn$ and intersects $\partial\bn$ transversely.
			\item[(2)] Any pair of the varieties $\{Z(I_j)\}$ does not intersect on $\partial\bn$. 
		\end{itemize}
		Then $I$ satisfies Hypothesis 1. Consequently, the submodule $\PP_I$ has the asymptotic stable division property with generating functions being polynomials of uniformly bounded degrees, and $\PP_I$ is $p$-essentially normal for all $p>n$.
	\end{thm}
	
	We will prove Theorem \ref{thm: Smooth implies ASD} in the remaining sections. 
	
	\medskip
	\noindent\textbf{Notations:} For the remainder of this paper, we reserve the notations $I$, $m_j$ and $I_j$ for the ones mentioned in Theorem \ref{thm: Smooth implies ASD}. Denote $Z=Z(I)$, $Z_j=Z(I_j)$, $\Z=Z\cap\bn$ and $\Z_j=Z_j\cap\bn$.

	As a preparation, we will prove a distance estimate for $Z(I)$ in this section.
	
	\begin{prop}\label{prop: distance estimate}
		For fixed $R>0$, 
		\begin{equation}\label{limit betaH}
		\sup\limits_{w\in D(z,R)\cap Z}\beta(w,T_zZ+z)\to 0,\quad |z|\to 1, z\in\Z.
		\end{equation}
		Here $T_zZ$ is the complex tangent space of $Z$ at $z$, viewed as a linear subspace of $\cn$.
	\end{prop}
	
	\begin{proof}
		Since $Z$ intersects $\partial\bn$ transversely, there exists $C_1>0$ such that, for any $z\in Z$ close enough to $\partial\bn$, there exists a non-zero vector $v_z\in T_zZ$ such that 
		\[
		|Q_z(v_z)|\leq C_1|\la v_z,z\ra|.
		\]
		It is also easy to see that there exists $C_2>0$ such that for $r>0$ sufficiently small and any $z\in Z$ close enough to $\partial\bn$, 
		\[
		\sup_{w\in Z\cap B(z,r)}\dist(w, T_zZ+z)\leq C_2 r^2.
		\]
		Here we use ``$\dist$'' to denote the Euclidean distance. Also, since $R>0$ is fixed, there exists $C_3>0$ such that for any $z, w\in\bn$ and $\beta(z,w)<R$,
		\[
		|P_z(w-z)|\leq C_3(1-|z|^2),~~~|w-z|^2\leq C_3(1-|z|^2).
		\]
		For any $z\in Z$ and $w\in D(z,R)\cap Z$, let $\lambda\in T_zZ$ be such that
		\[
		\dist(w, \lambda+z)=\dist(w, T_zZ+z).
		\]
		Then
		\[
		|w-(\lambda+z)|\leq C_2|w-z|^2\leq C_2C_3(1-|z|^2).
		\]
		Let $\lambda'=\lambda-\frac{\la \lambda+z-w,z\ra}{\la v_z,z\ra}v_z$. Then $\lambda'\in T_zZ$. We need to estimate $\beta(\lambda'+z,w)$. 
		Noticing that $\la\lambda'+z,z\ra=\la w,z\ra$, we have
		\[
		|\varphi_z(\lambda'+z)-\varphi_z(w)|=\frac{(1-|z|^2)^{1/2}|Q_z(w-\lambda')|}{|1-\la w,z\ra|}\lesssim\frac{|Q_z(w-\lambda')|}{(1-|z|^2)^{1/2}}.
		\]
		Since
		\begin{eqnarray*}
			|Q_z(w-\lambda')|&\leq&|Q_z(w-\lambda)|+\frac{|\la\lambda+z-w,z\ra|}{|\la v_z,z\ra|}|Q_z(v_z)|\\
			&=&|Q_z(w-(\lambda+z))|+\frac{|\la\lambda+z-w,z\ra|}{|\la v_z,z\ra|}|Q_z(v_z)|\\
			&\leq&(1+C_1)|w-(\lambda+z)|\\
			&\leq&(1+C_1)C_2C_3(1-|z|^2),
		\end{eqnarray*}
	we have
		\[
		|\varphi_z(\lambda'+z)-\varphi_z(w)|\lesssim(1-|z|^2)^{1/2}.
		\]
		Since $w\in D(z,R)$, from the above inequality, for $z$ close enough to $\partial\bn$, $\varphi_z(w), \varphi_z(\lambda'+z)$ fall into a compact subset of $\bn$. Hence
		\[
		\beta(\lambda'+z, w)=\beta(\varphi_z(\lambda'+z),\varphi_z(w))\approx|\varphi_z(\lambda'+z)-\varphi_z(w)|\lesssim(1-|z|^2)^{1/2}.
		\]
		This completes the proof.
	\end{proof}
	
	It is convenient to consider the following modified version of tangent space, because it is invariant under the M\"{o}bius transform $\varphi_z$.
	\begin{defn}
		For $z\in\Z$, let us define the \emph{normal tangent cone} $NT_zZ$ to be 
		\[
		NT_zZ=T_zZ\cap z^\perp+\CC z.
		\]
	\end{defn}
	Note that $z\in NT_zZ$.
	\begin{lem}\label{lem: NT close to T}
		For any $R>0$, we have
		\[
		\sup_{w\in (T_zZ+z)\cap D(z,R)}\beta(w, NT_zZ)\to0,\quad z\in\Z, |z|\to1.
		\]
	\end{lem}
	\begin{proof}
		For $z\in\Z$ close enough to $\partial\bn$, let $z_0=P_{T_zZ}z$. Since $\Z$ intersects $\partial\bn$ transversely, we can assume that $|z_0|\geq\epsilon$ for some $0<\epsilon<1$. 
		
		Suppose $w\in T_zZ$ and $\beta(w+z,z)<R$. Let $\lambda=P_zw+w-\la w,z\ra\frac{z_0}{|z_0|^2}$. Then $\lambda\in NT_zZ$ and $\la w,z\ra=\la\lambda,z\ra$. So $\varphi_z(\lambda +z)$ is well defined. Therefore
		\[
		|\varphi_z(z+w)-\varphi_z(z+\lambda)|=\bigg|\frac{(1-|z|^2)^{1/2}\frac{\la w,z\ra}{|z_0|^2}Q_z(z_0)}{1-\la z+w, z\ra}\bigg|\lesssim\frac{\frac{1-|z|^2}{|z_0|}}{(1-|z|^2)^{1/2}}\leq\epsilon^{-1}(1-|z|^2)^{1/2}\to0,
		\]
		where the first inequality follows from the fact that $|\la w,z\ra|=|\la w+z-z, z\ra|\lesssim 1-|z|^2$.
		This completes the prove.
	\end{proof}
	Combining Proposition \ref{prop: distance estimate} Lemma \ref{lem: NT close to T}, we have the following lemma.
	\begin{lem}\label{lem: Z close to NT}
		For any $R>0$, we have
		\[
		\sup_{w\in Z\cap D(z,R)}\beta(w, NT_zZ)\to0,\quad z\in\Z, |z|\to1.
		\]
	\end{lem}
	
	\section{A Covering Lemma}\label{section: a covering lemma}
	As a first step to proving Theorem \ref{thm: Smooth implies ASD}, for arbitrarily large $R>0$, we will  construct covers $\{E_i\}_{i\in\Lambda}$ and $\{F_i\}_{i\in\Lambda}$. In \cite{Xia18}, Xia constructed covers of $\bn$ with the bounded overlap condition (2) in Hypothesis 1. Our construction will follow the general framework of \cite{Xia18}, but with additional requirements. First, our covers need to take the variety $\Z$ into account. For this reason, we will construct the covers in two steps, first on $\Z$ and then on $\bn$. Second, to construct the decompositions on each cover set, we need the sets to be rotation invariant in certain directions. This property will be used in the proof of Lemma \ref{lem: general decomposition lemma}.

	For an arbitrarily large $R>0$, choose $s'$ so that 
	\[
	(s'-1)\log2\leq16R<s'\log2,
	\]
	and then take $s=10s'$. For a positive integer $k$, define 
	\[
	r_k=\sqrt{1-2^{-2sk}},~~~r_k'=\sqrt{1-2^{-2s'k}}.
	\]
	Write $\Z_k=r_kS\cap\Z$, where $r_k S=\{r_k x: x\in S\}$. Let $\LL_k\subset\Z_k$ be maximal with respect to the following property: if $u_i, u_j\in\LL_k$ and $u_i\neq u_j$, then $\ds(u_i, u_j)>s2^{-sk-2}$.
	
	For $u\in\LL_k$, set 
	\[
	\E_u=\{z\in\bn: 1-|z|^2>2^{-5s}, |1-\la z,u\ra|>s^{-2}, \frac{1-|z|^2}{|1-\la z,u\ra|^2}<2^{-2s}\},
	\]
	\[
	\F_u=\{z\in\bn: 1-|z|^2>2^{-6s}, |1-\la z,u\ra|>\frac{1}{2}s^{-2}, \frac{1-|z|^2}{|1-\la z,u\ra|^2}<2^{-s}\},
	\]
	and
	\[
	E_u=\varphi_u(\E_u), \quad F_u=\varphi_u(\F_u).
	\]
	
	\medskip
	Let
	\[
	\mathcal{O}=\{w\in\bn: \beta(w,\Z)>5s'\log2\}.
	\]
	Take $S_k=r_k'S\cap\mathcal{O}$. Let $\LL_k'\subseteq S_k$ be maximal with respect to the following property: if $v_i, v_j\in S_k$ and $v_i\neq v_j$, then $\ds(v_i, v_j)>s'2^{-s'k-2}$. For $v\in\LL_k'$, let
	\[
	\E_v'=\{z\in\bn: 1-|z|^2>2^{-5s'}, |1-\la z,v\ra|>s'^{-2},\frac{1-|z|^2}{|1-\la z,u\ra|^2}<2^{-2s'}\},
	\]
	and
	\[
	E_v'=\varphi_v(\E_v').
	\]
	Set
	\[
	\F_v'=\E_{v3}',\quad F_v'=E_{v3}',
	\]
	where for any set $A\subset\bn$ and $j\in\mathbb{N}$, we denote $A_j=\{w\in\bn: \beta(w, A)<jR\}$.
	Let $K$ be a positive integer determined later, and $\LL=\cup_{k\geq K}\LL_k$, $\LL'=\cup_{k\geq10K}\LL_k'$.

	 	\medskip

	Let us establish some basic properties of the sets.
	\begin{lem}\label{lem: E_u properties}
		For $k$, $s'$ large enough and $u\in\LL_k$, $v\in\LL_k'$, the following hold.
		\begin{itemize}
			\item[(1)]
			\[
			E_u=\{w\in\bn: 1-|w|^2\in(2^{-2(k+3)s},2^{-2(k+1)s}), d(u,w)<s2^{-ks}, 1-|\varphi_u(w)|^2>2^{-5s}\},
			\]
			\[
			E_v'=\{w\in\bn: 1-|w|^2\in(2^{-2(k+3)s'},2^{-2(k+1)s'}), d(v,w)<s'2^{-ks'}, 1-|\varphi_v(w)|^2>2^{-5s'}\},
			\]
			\[
			F_u=\{w\in\bn: 1-|w|^2\in(2^{-2(k+4)s},2^{-2ks-s}), d(u,w)<\sqrt{2}s2^{-ks}. 1-|\varphi_u(w)|^2>2^{-6s}\},
			\]
			\item[(2)]
			\[
			E_u\subseteq D(u, 3s\log2),\quad E_v'\subseteq D(v, 3s'\log2),\quad F_u\subset D(u, 4s\log2).
			\]
			\item[(3)]
			\[
			\{w\in\bn: \beta(w, E_u)<\frac{s}{4}\log2\}\subset F_u.
			\]
			Thus $E_{u4}\subset F_u$.
			\item[(4)]
		\[
		E_u\supseteq\{w\in\bn: 1-|w|^2\in[2^{-2(k+2)s}, 2^{-2(k+1)s}), \ds(u,w)<s2^{-ks-1}\},
		\]
		\[
		E_v'\supseteq\{w\in\bn: 1-|w|^2\in[2^{-2(k+2)s'}, 2^{-2(k+1)s'}), \ds(v,w)<s'2^{-ks'-1}\}.
		\]
			\item[(5)] There exists a constant $M>0$, depending only on $n$, such that any $z\in\bn$ belongs to at most $M$ of the sets $\{F_u\}_{u\in\LL}\cup\{F_v'\}_{v\in\LL'}$.
		\end{itemize}
	\end{lem}
	
	\begin{proof}
		First, we prove (1). Suppose $u\in\LL_k$. By definition, $w\in E_u$ if and only if $\varphi_u(w)\in\E_u$, i.e.,
		\[
		1-|\varphi_u(w)|^2>2^{-5s},\quad |1-\la \varphi_u(w), u\ra|>s^{-2},\quad \frac{1-|\varphi_u(w)|^2}{|1-\la \varphi_u(w), u\ra|^2}<2^{-2s}.
		\]
		Since $1-|u|^2=2^{-2sk}$ and by Lemma \ref{basic about varphi} (1), the conditions above are equivalent to 
		\[
		1-|\varphi_u(w)|^2>2^{-5s},\quad |1-\la w,u\ra|<s^{2}2^{-2sk},\quad 1-|w|^2<2^{-2(k+1)s}.
		\]
		Also, by Lemma \ref{basic computations} (2),
		\[
		1-|w|^2=1-|\varphi_u\varphi_u(w)|^2>\frac{1}{4}(1-|u|^2)(1-|\varphi_u(w)|^2)>\frac{1}{4}2^{-2sk-5s}>2^{-2(k+3)s}.
		\]
		Thus we have 
		\[
		E_u=\{w\in\bn: 1-|w|^2\in(2^{-2(k+3)s},2^{-2(k+1)s}), d(u,w)<s2^{-ks}, 1-|\varphi_u(w)|^2>2^{-5s}\}.
		\]
		The proof of (1) for $E_v'$ and $F_u$ is similar. 
		
		The proof of (2) is a straightforward application of Lemma \ref{lem: beta rho} to the estimate of $1-|\varphi_u(w)|^2$ above.
		
		To prove (3), suppose $w\in\bn$ and $\beta(w, E_u)<\frac{s}{4}\log2$. Choose $z\in E_u$ such that $\beta(w, z)<\beta(w, E_u)+\log2<\frac{s}{4}\log2+\log2$. Then by Lemma \ref{lem: beta rho},
		\[
		1-|\varphi_w(z)|^2\geq e^{-2\beta(z,w)}>e^{-\frac{s}{2}\log2-2\log2}=2^{-\frac{s}{2}-2}.
		\]
		By Lemma \ref{basic computations} (2) and part (1) proved above,
		\begin{equation}\label{eqn: 5.1.1}
		1-|w|^2>\frac{1}{4}(1-|z|^2)(1-|\varphi_w(z)|^2)>\frac{1}{4}\cdot2^{-2(k+3)s}\cdot2^{-\frac{s}{2}-2}>2^{-2(k+4)s}
		\end{equation}
		and
		\begin{equation}\label{eqn: 5.1.2}
		1-|w|^2<4\frac{1-|z|^2}{1-|\varphi_w(z)|^2}<4\frac{2^{-2(k+1)s}}{2^{-2-\frac{s}{2}}}<2^{-2ks-s}.
		\end{equation}
		
		Also,
		\[
		d^4(z,w)=|1-\la z,w\ra|^2=\frac{(1-|z|^2)(1-|w|^2)}{1-|\varphi_w(z)|^2}<\frac{2^{-2(k+1)s}\cdot 2^{-2ks-s}}{2^{-\frac{s}{2}-2}}<2^{-4ks-2s}.
		\]
		Therefore $d(z,w)<2^{-ks-\frac{s}{2}}$. Hence 
		\begin{equation}\label{eqn: 5.1.3}
		d(u,w)\leq d(u,z)+d(z,w)<s2^{-ks}+2^{-ks-\frac{s}{2}}<\sqrt{2}s2^{-ks}.
		\end{equation}
		 Finally, by Lemma \ref{lem: beta rho}, 
		\[
		\beta(z,u)<\log2-\frac{1}{2}\log(1-|\varphi_u(z)|^2)<\log2-\frac{1}{2}\log2^{-5s}=(\frac{5}{2}s+1)\log2.
		\]
		Therefore
		\[
		\beta(w,u)\leq\beta(w,z)+\beta(z,u)<\frac{s}{4}\log2+\log2+(\frac{5}{2}s+1)\log2<3s\log2
		\]
		and
		\begin{equation}\label{eqn: 5.1.4}
		1-|\varphi_u(w)|^2>e^{-2\cdot3s\log2}=2^{-6s}.
		\end{equation}
		From inequalities \eqref{eqn: 5.1.1}\eqref{eqn: 5.1.2}\eqref{eqn: 5.1.3}\eqref{eqn: 5.1.4}, we have
		\[
		\{w\in\bn: \beta(w, E_u)<\frac{s}{4}\log2\}\subset F_u.
		\]
		This proves (3).

			Suppose $1-|z|^2\in[2^{-2(k+2)s}, 2^{-2(k+1)s})$ and $\ds(u,z)<s2^{-ks-1}$. By Lemma \ref{ds and d},
		\begin{eqnarray*}
			d^2(u,z)&<&\ds^2(u,z)+(1-|z|^2)+(1-|u|^2)\\
			&<&s^22^{-2ks-2}+2^{-2(k+1)s}+2^{-2ks}\\
			&<&s^22^{-2ks-1}.
		\end{eqnarray*}
		We have
		\[
		1-|\varphi_u(z)|^2=\frac{(1-|u|^2)(1-|z|^2)}{|1-\la u,z\ra|^2}>\frac{2^{-2ks}2^{-2(k+2)s}}{s^42^{-4ks-2}}=s^{-4}2^{-4s+2}>2^{-5s}.
		\]
		By (1), we have $z\in E_u$. This proves (4).
		
		Finally, we prove (5). Suppose $z\in\bn$. Let $k$ be the positive integer such that $2^{-2(k+1)s}<1-|z|^2\leq2^{-2ks}$. If $z\in F_u$ and $u\in\LL_l$, then by (1), $2^{-2(l+4)s}<1-|z|^2<2^{-2ls-s}$. Thus $k-3\leq l\leq k$. Since $z\in F_u$ and by Lemma \ref{ds and d} (2), $\ds(z,u)<\sqrt{2}d(z,u)<2s2^{-ls}$. By our construction, if $u_1\neq u_2$, $u_1, u_2\in\LL_l$, then $\ds(u_1,u_2)>s2^{-sl-2}$. By Lemma \ref{Q delta size}, the number of $u\in\LL_l$ such that $z\in F_u$ does not exceed $C\bigg(\frac{2s2^{-ls}+\frac{s}{2}2^{-sl-2}}{\frac{s}{2}2^{-sl-2}}\bigg)^{2n}<C2^{10n}$, where $C>0$ is a constant. Thus the number of $u\in\LL$ such that $z\in F_u$ does not exceed  $C2^{10n+2}$. Replacing $s$ with $s'$ in the proof of (3) will give us that the set $\{w\in\bn: \beta(w, E_v')<\frac{s'}{4}\log2\}$ is contained in
		\[
		\{w\in\bn: 1-|w|^2\in(2^{-2(k+4)s'}, 2^{-2ks'-s'}), d(v,w)<\sqrt{2}s'2^{-ks'}, 1-|\varphi_v(w)|^2>2^{-6s'}\}.
		\]
		Thus $E_{v3}'$ is contained in the right hand side of the above. Thus similar proof will show that $z$ belongs to at most  $C2^{10n}$ of the sets $F_v'=E_{v3}'$.
		Altogether, the total number of $u\in\LL$ and $v\in\LL'$ such that $z\in F_u$ or $z\in F_v'$ does not exceed $M:=C2^{10n+3}$. This completes the proof.
	\end{proof}
	
	\begin{lem}\label{lem: z is close to u}
		For $k$ large enough and $z\in\Z$ with $1-|z|^2\in[2^{-2(k+3)s}, 2^{-2ks}]$, there exists $u\in\LL_k$ such that 
		\[
		\ds(u,z)<\frac{s}{3}2^{-sk}.
		\]
	\end{lem}
	
	\begin{proof}
		By assumption, $Z$ intersects $\partial\bn$ transversely and has no singular points on $\partial\bn$. It is easy to see that for $z\in Z$ close enough to $\partial\bn$, the orthogonal projection onto $NT_zZ$, when restricted to $Z$, is a one-sheeted analytic cover in an Euclidean neighborhood $U$ of $z$. Consider such a specific $z$. Choose an orthonormal basis $\{e_1,\cdots,e_n\}$ of $\cn$ such that $e_1=\frac{z}{|z|}$ and $NT_zZ=\mathrm{span}\{e_1,\cdots,e_d\}$, where $d$ is the dimension of $Z$ at $z$. There is a vector-valued holomorphic  function $a$ such that $\zeta\in Z\cap U$ if and only if $\zeta=(\zeta', a(\zeta'))$ under the new basis, where $\zeta'=(\zeta_1,\cdots,\zeta_d)$, under the new basis. 
		
		We have $\frac{z'}{|z|}=(1,0,\cdots,0)\in\CC^d$ and $a(z')=0$ under the new basis. Consider the paths
		\[
		\gamma:[0,1]\to\CC^d, t\mapsto (1-t)z'+t\cdot r_{k-1}\frac{z'}{|z|},
		\]
		\[
		\Gamma(t)=(\gamma(t), 0)\subseteq NT_zZ,
		\]
		and
		\[
		\Lambda(t)=(\gamma(t), a(\gamma(t)))\subseteq\Z.
		\]
		Note that $\Lambda(0)=\Gamma(0)=z$, so $1-|\Lambda(0)|^2\leq2^{-2sk}$. On the other hand, by the standard inverse function theorem, there is a constant $C>0$ such that for $z\in\Z$ close enough to $\partial\bn$ and $\lambda', w'$ in a sufficiently small Euclidean neighborhood of $z'$, we have 
		\[
		|a(\lambda')-a(w')|\leq C|\lambda'-w'|.
		\] 
		So when $k$ is large enough,
		\begin{eqnarray*}
			1-|\Lambda(1)|^2&=&1-|\gamma(1)|^2-|a(\gamma(1))|^2\\
			&\geq&1-|\gamma(1)|^2-C^2|\gamma(1)-\gamma(0)|^2\\
			&\geq&2^{-2s(k-1)}-C^22^{-4sk}\\
			&\geq&2^{-2sk}.
		\end{eqnarray*}
		By the intermediate value theorem, there exists $t_0\in[0,1]$ such that
		\[
		1-|\Lambda(t_0)|^2=2^{-2sk}.
		\]
		Denote $z_0=\Lambda(t_0)$. Then $z_0\in\Z_k$. From our construction of $\LL_k$, there exists a $u\in\LL_k$ such that
		\[
		\ds(u,z_0)\leq s2^{-sk-2}.
		\]
		Also, since $\frac{z}{|z|}=\frac{\Gamma(t_0)}{|\Gamma(t_0)|}$,
		\begin{eqnarray*}
			\ds^2(z,z_0)&=&\ds^2(\Gamma(t_0),\Lambda(t_0))\\
			&<&2d^2(\Gamma(t_0),\Lambda(t_0))\\
			&=&2(1-|\gamma(t_0)|^2)\\
			&\leq&2(1-|\Lambda(t_0)|^2)+2|a(\gamma(t_0))|^2\\
			&\leq&2\cdot2^{-2sk}+2C^22^{-4sk}\\
			&\leq&2^{-2sk+2}.
		\end{eqnarray*}
		Thus
		\[
		\ds(z, z_0)<2^{-sk+1}.
		\]
		So
		\[
		\ds(z,u)\leq\ds(z,z_0)+\ds(z_0,u)<2^{-sk+1}+s2^{-sk-2}<\frac{s}{3}2^{-sk}.
		\]
		This completes the proof.
	\end{proof}
	
	\begin{lem}\label{lem: Eu contain Bergman nbhd of Z}
		For $s$ and $K$ large enough,
		\[
		\{w\in\bn:~1-|w|^2<2^{-2s(K+1)}, \beta(w,\Z)<(s-1)\log2\}\subset\bigcup_{k=K}^\infty\bigcup_{u\in\LL_k}E_u.
		\]
	\end{lem}
	
	\begin{proof}
		Suppose $w\notin\bigcup_{l=K}^\infty\bigcup_{u\in\LL_l}E_u$, and $1-|w|^2\in[2^{-2s(k+2)},2^{-2s(k+1)})$ for some $k\geq K$. For any $z\in\Z$, we want to show  that $\beta(z,w)\geq (s-1)\log2$.
		
		If $1-|z|^2\notin(2^{-2s(k+3)},2^{-2sk})$, then
		\[
		1-|\varphi_z(w)|^2<\min\{4\frac{1-|z|^2}{1-|w|^2}, 4\frac{1-|w|^2}{1-|z|}\}<2^{-2s+2}.
		\]
		Therefore
		\[
		\beta(z,w)>-\frac{1}{2}\log(1-|\varphi_z(w)|^2)>(s-1)\log2.
		\]
		
		If $1-|z|^2\in(2^{-2s(k+3)},2^{-2sk})$, by Lemma \ref{lem: z is close to u}, there exists $u\in\LL_k$ such that 
		\[
		\ds(u,z)<\frac{s}{3}2^{-sk}.
		\]
		
		Since $w\notin E_u$, by Lemma \ref{lem: E_u properties} (4),
		\[
		\ds(w,u)\geq s2^{-sk-1}.
		\]
		So
		\[
		\ds(w,z)\geq\ds(w,u)-\ds(z,u)>s2^{-sk-1}-\frac{s}{3}2^{-sk}=\frac{s}{6}2^{-sk}.
		\]
		Hence
		\[
		d^2(w,z)>\frac{1}{2}\ds^2(w,z)>\frac{s^2}{100}2^{-2sk}.
		\]
		Therefore we have
		\[
		1-|\varphi_z(w)|^2=\frac{(1-|z|^2)(1-|w|^2)}{|1-\la z,w\ra|^2}<10^4s^{-4}2^{-2s}<2^{-2s}.
		\]
		So
		\[
		\beta(z,w)>-\frac{1}{2}\log(1-|\varphi_z(w)|^2)>s\log2.
		\]
		This completes the proof.
	\end{proof}
	
	Lemma \ref{lem: Eu contain Bergman nbhd of Z} shows that, close to the boundary, the sets $E_u$ cover a Bergman neighborhood of the variety $\Z$. In fact, we have the following lemma.
	\begin{lem}\label{lem: Eu Ev are coverings}
		For $s'$ and $K$ large enough, the following are true.
		\begin{itemize}
			\item[(1)] For any $k\geq 10K$ and $v\in\LL_k'$,
			\[
			\{w\in\bn:~\beta(w, E_v')<s'\log2\}\cap\Z=\emptyset.
			\]
			\item[(2)] 
			\[
			\bigg(\bigcup_{k=10 K}^\infty\bigcup_{v\in\LL_k'} E_v'\bigg)\bigcup\bigg(\bigcup_{k= K}^\infty\bigcup_{u\in\LL_k} E_u\bigg)\supset\mathbb{B}_{\delta},
			\]
			where $1-\delta^2=2^{-2s(K+1)}$.
		\end{itemize}
	\end{lem}
	\begin{proof}
	For $k\geq 10K$ and $v\in\LL_k'$, by Lemma \ref{lem: E_u properties} (2), 
		\[
		\{w\in\bn:~\beta(w, E_v')<s'\log2\}\subset D(v, 4s'\log2).
		\]
		If $\{w\in\bn:~\beta(w, E_v')<s'\log2\}\cap\Z\neq\emptyset$, then $\beta(v, \Z)<4s'\log2$,
		which contradicts the definition of $\LL_k'$. This proves (1).
		
		Suppose $w\in\mathbb{B}_\delta$. Then $1-|w|^2<2^{-2s(K+1)}$. If $w\notin\bigcup_{l= K}^\infty\bigcup_{u\in\LL_l} E_u$, then by Lemma \ref{lem: Eu contain Bergman nbhd of Z}, 
		\[
		\beta(w,\Z)\geq(s-1)\log2>9s'\log2.
		\]
		Suppose $1-|w|^2\in[2^{-2s'(k+2)}, 2^{-2s'(k+1)})$. Then $k>10K$. Let $w_0=\frac{r_k'}{|w|}w$. Then
		\[
		1-|\varphi_w(w_0)|^2=\frac{(1-|w|^2)(1-|w_0|^2)}{|1-\la w,w_0\ra|^2}>\frac{1}{4}\frac{1-|w|^2}{1-|w_0|^2}\geq\frac{2^{-2s'(k+2)}}{4\cdot2^{-2s'k}}=2^{-4s'-2}.
		\]
		By Lemma \ref{lem: beta rho}, $\beta(w, w_0)<\log2-\frac{1}{2}\log2^{-4s'-2}<3s'\log2$. Therefore
		\[
		\beta(w_0,z)\geq\beta(w,z)-\beta(w,w_0)>6s'\log2.
		\]
		Thus $w_0\in\mathcal{O}\cap r_k'S$.
		 By construction, there exists $v\in\LL_k'$ such that $\ds(w_0,v)<s'2^{-s'k-2}$. Therefore $\ds(w, v)=\ds(w_0, v)<s'2^{-s'k-2}$. By Lemma \ref{lem: E_u properties} (4), $w\in E_v'$. This completes the proof. 
	\end{proof}
	
	\section{Local Decomposition Formulas}\label{section: local decomposition formulas}
	
To show that the $I$ in Theorem \ref{thm: Smooth implies ASD} satisfies Hypothesis 1, we need to construct decompositions $\sum\pij\fij$ on the sets $E_{u3}$(or $E_{v3}'$), for $f\in I$. It is more convenient to construct the decompositions on their images under a M\"{o}bius transform.
	
	\begin{lem}\label{lem: decomposition on Fu}
		There exist a positive integer $N$ and a constant $C>0$, depending on $I$, such that the following hold. For any $R>0$ sufficiently large, let $\{E_u\}_{u\in\LL}$, $\{F_u\}_{u\in\LL}$, $\{E_v'\}_{v\in\LL'}$ and $\{F_v'\}_{v\in\LL'}$ be as in the beginning of Section \ref{section: a covering lemma}. Then there exists a positive integer $K$ such that the following hold.
		\begin{itemize}
			\item[(1)] For each $u\in\LL_k$, $k\geq K$, there is a finite set of polynomials $\{p_l\}\subset I$, $\sup_l\deg p_l\leq N$, and linear maps $I\to \mathrm{Hol}(E_{u3})$, $f\mapsto f_l$ satisfying the following inequalities.
			\begin{itemize}
				\item[(i)]$\int_{E_{u3}}\big|\sum_l p_lf_l-f\big|^2dv\leq C\epsilon_{R}^2\int_{F_u}|f|^2dv$.
				\item[(ii)]$\int_{E_{u3}}\big(\sum_l|p_lf_l|\big)^2dv\leq C\int_{F_u}|f|^2dv$.
			\end{itemize} 
		Here $\epsilon_R=\epsilon_{R,0}$ is defined in Lemma \ref{lem: norm estimates for integral kernels}.
			\item[(2)] For each $v\in\LL'_k$, $k\geq 10K$, there is a $p\in I$ such that $p$ is non-vanishing on $F_v'$ and $\deg p\leq N$. 
		\end{itemize}
	\end{lem}

	Some preparations are needed for the proof of Lemma \ref{lem: decomposition on Fu}. 
	\begin{lem}\label{lem: xy decomposition}
		For any positive integer $m, k$, and $x, y\in\CC$, $x, y\neq0$, there are constants $\{c_j\}_{j=0}^{m-1}, \{d_j\}_{j=0}^{k-1}$, such that
		\[
		\frac{1}{x^k}=\sum_{j=0}^{m-1}c_j\frac{(x-y)^j}{y^{k+j}}+\sum_{j=0}^{k-1}d_j\frac{(x-y)^m}{x^{k-j}y^{m+j}}.
		\]
	\end{lem}
	\begin{proof}
		We prove by induction on $m$. For $m=1$ and any $k$,
		\[
		\frac{1}{x^k}=\frac{1}{y^k}+(\frac{1}{x^k}-\frac{1}{y^k})
		=\frac{1}{y^k}-\sum_{j=0}^{k-1}\frac{x^jy^{k-1-j}}{x^ky^k}(x-y)
		=\frac{1}{y^k}-\sum_{j=0}^{k-1}\frac{(x-y)}{x^{k-j}y^{1+j}}.
		\]
		Suppose the equation holds for $m-1$ and any $k$. Then
		\begin{eqnarray*}
			\frac{1}{x^k}&=&\sum_{j=0}^{m-2}c_j'\frac{(x-y)^j}{y^{k+j}}+\sum_{j=0}^{k-1}d_j'\frac{(x-y)^{m-1}}{x^{k-j}y^{m-1+j}}\\
			&=&\sum_{j=0}^{m-2}c_j'\frac{(x-y)^j}{y^{k+j}}+\sum_{j=0}^{k-1}d_j'\frac{(x-y)^{m-1}}{y^{m-1+j}}\big(\frac{1}{y^{k-j}}-\sum_{i=0}^{k-j-1}\frac{(x-y)}{x^{k-j-i}y^{1+i}}\big)\\
			&=&\sum_{j=0}^{m-1}c_j\frac{(x-y)^j}{y^{k+j}}+\sum_{l=0}^{k-1}d_l\frac{(x-y)^m}{x^{k-l}y^{m+l}}.
		\end{eqnarray*}
		This completes the proof.
	\end{proof}
	
	The following lemma is elementary.
	\begin{lem}\label{lem: sphereical shell}
		Suppose $0<a<b$ and $d$ is a positive integer. Write $B_a^b=\{z\in\CC^d: a<|z|<b\}$. Then for any $f\in\mathrm{Hol}(b\mathbb{B}_d)$ we have 
		\begin{itemize}
			\item[(1)]
			\[
		\int_{b\mathbb{B}_d}|f|^2dv\leq\frac{b}{b-a}\int_{B_a^b}|f|^2dv.
		\]	
		\item[(2)]For any multi-index $\alpha=(\alpha_1,\cdots,\alpha_d)$,
		\[
		\int_{B_a^b}f(w)\overline{w^\alpha}dv(w)=c_{a,b}\partial^\alpha f(0).
		\]
		Here $c_{a,b,\alpha}=\frac{1}{\alpha !}\int_{B_a^b}|w^\alpha|^2dv(w).$
		\end{itemize}
	\end{lem}
	
	\begin{lem}\label{lem: general decomposition lemma}
		Let $d, m$ be positive integers, $0<d<n$. For $R>0$ and a positive integer $k$ that are large enough, the following hold. Let $s=10s'$, where $s'$ is the unique integer satisfying $(s'-1)\log2\leq16R<s'\log2$. Write $r=\sqrt{1-2^{-2ks}}$, $u=(r,0,\cdots,0)$, and define $E:=\E_u$, $F:=\F_u$ as in the beginning of Section \ref{section: a covering lemma}, i.e.,
		\[
		E=\{z\in\bn: 1-|z|^2>2^{-5s}, |1-\la z,u\ra|>s^{-2}, \frac{1-|z|^2}{|1-\la z,u\ra|^2}<2^{-2s}\},
		\]
		and
		\[
		F=\{z\in\bn: 1-|z|^2>2^{-6s}, |1-\la z,u\ra|>\frac{1}{2}s^{-2}, \frac{1-|z|^2}{|1-\la z,u\ra|^2}<2^{-s}\}.
		\]
		
		Write $U'=\pi(F)$, where $\pi$ is the projection $\pi: \CC^n\to\CC^d, z\mapsto z':=(z_1,\cdots,z_d)$.	Then there exist $\delta>0$, depending only on $d, m, R$, and $C>0$, depending only on $d, m$, with the following properties. 
		
		Suppose $a: U'\to\CC^{n-d}$ is holomorphic and $|a(z')|<\delta$, $\forall z'\in U'$. Write $A=\{(z',a(z')): z'\in U'\}$. Let $\{P_\alpha=P_{\pi|_{mA},\alpha}\}$ be the set of canonical defining functions with respect to $\pi|_{mA}$.  Here $\alpha=(\alpha_{d+1},\cdots,\alpha_n)$, $|\alpha|=m$. Then there exist linear maps $f\mapsto f_\alpha$, where $f\in\mathrm{Hol}(\bn)$, $f|_{mA}=0$, and $f_\alpha\in\mathrm{Hol}(E_3)$, with the following properties.
		\begin{itemize}
			\item[(1)]
			\[
			\int_{E_3}|f-\sum_{|\alpha|=m}P_\alpha f_\alpha|^2dv\leq C\epsilon_{R}^2\int_{F}|f|^2dv.
			\]
			\item[(2)] For any $\alpha, \beta, |\alpha|=|\beta|=m$,
			\[
			\int_{E_3}|P_\beta f_\alpha|^2dv\leq C\int_{F}|f|^2dv.
			\]
		\end{itemize}
		Here $\epsilon_{R}=\epsilon_{R,0}$ is the constant defined in Lemma \ref{lem: norm estimates for integral kernels}.
		
	\end{lem}
	\begin{proof}
		To simplify notations, we write $\pi(z)=(z',0)$ and $p(z)=(z',a(z'))$.

		We will construct the functions $f_\alpha$ by decomposing the reproducing kernel $K_w(z)$. For $z\in F$ and $w\in\bn$, applying Lemma \ref{lem: xy decomposition} for $x=1-\la z,w\ra$, $y=1-\la p(z),w\ra$, we get
		\begin{eqnarray*}
			K_w(z)&=&\sum_{j=0}^{m-1}c_j\frac{\la a(z')-z'',w''\ra^j}{(1-\la p(z),w\ra)^{n+1+j}}+\sum_{j=0}^nd_j\frac{\la a(z')-z'',w''\ra^m}{(1-\la z,w\ra)^{n+1-j}(1-\la p(z),w\ra)^{m+j}}\\
			&=&I(z,w)+II(z,w).
		\end{eqnarray*}
		Notice that $\la a(z')-z'', w''\ra^m=P_{\pi|_{mA}}(z,w'')=\sum_{|\alpha|=m}P_\alpha(z)\overline{w''^\alpha}$. Write 
		\[
		G(z,w)=\sum_{j=0}^nd_j\frac{1}{(1-\la z,w\ra)^{n+1-j}(1-\la p(z),w\ra)^{m+j}}.
		\]
		Then $II(z,w)=\sum_{|\alpha|=m}G(z,w)\overline{w''^\alpha}P_\alpha(z)$. Let
		\[
		f_\alpha(z)=c_R^{-1}\int_{F_1}f(w)G(z,w)\overline{w''^\alpha}dv(w).
		\]
		Here $c_R=c_{R,0}$ is the constant in Lemma \ref{lem: norm estimates for integral kernels}.
		Then 
		\begin{equation}\label{eqn: 6.4.1}
		(P_\beta f_\alpha)(z)=c_R^{-1}\int_{F_1}f(w)G(z,w)\overline{w''^\alpha}P_\beta(z)dv(w).
		\end{equation}
		 We will show that the functions satisfy inequalities (1) and (2).
		
		First, we prove inequality (2). Notice that by definition, $1-|z|^2>2^{-6s}, \forall z\in F$. If $\delta<2^{-6s-2}$ and $|a(\xi')|\leq\delta$, $\forall\xi'\in U'$, then for any $z\in E_3$, $w\in F$, we have
		\[
		|1-\la p(z),w\ra|=|1-\la z',w'\ra-\la a(z'),w''\ra|\geq|1-\la z',w'\ra|-|a(z')|>|1-\la z',w'\ra|-2^{-6s-2}\geq\frac{1}{2}|1-\la z',w'\ra|.
		\]
		Here the last inequality holds because 
		\[
		|1-\la z',w'\ra|\geq1-|w|\geq\frac{1}{2}(1-|w|^2)>2^{-6s-1}.
		\]
	Thus
		\[
		|G(z,w)|\lesssim\sum_{j=0}^n\frac{1}{|1-\la z,w\ra|^{n+1-j}|1-\la z',w'\ra|^{m+j}}.
		\]
		For each pair $\alpha, \beta, |\alpha|=|\beta|=m$ and $z\in E_3, w\in F$, since 
		\[
		|\overline{w''^\alpha}|\leq|w|^{m/2}\leq(1-|w'|^2)^{m/2}\lesssim|1-\la z',w'\ra|^{m/2},
		\]
		and  
		\[
		|P_\beta(z)|\leq|a(z')-z''|^{m/2}\lesssim|1-\la z',w'\ra|^{m/2},
		\]
		we have
		\[
		|G(z,w)\overline{w''^\alpha}P_\beta(z)|\lesssim\sum_{j=0}^n\frac{1}{|1-\la z,w\ra|^{n+1-j}|1-\la z',w'\ra|^{j}}\lesssim\frac{1}{|1-\la z,w\ra|^{n+1}}+\frac{1}{|1-\la z',w'\ra|^{n+1}}.
		\]
			By \eqref{eqn: 6.4.1} and Lemma \ref{lem: norm estimates for integral kernels}, we obtain inequality (2).

		Next, we prove (1). For any $z\in E_3$, by Lemma \ref{lem: E_u properties} (2), (3), $D(z,R)\subset F$. By Lemma \ref{lem: norm estimates for integral kernels} (3),
		\begin{equation}\label{eqn: 1}
		\int_{F}f(w)K_w(z)dv(w)-c_R f(z)=\int_{F\backslash D(z,R)}f(w)K_w(z)dv(w).
		\end{equation}
		Thus
		\begin{eqnarray}\label{eqn: (1)}
		&&f(z)-\sum_\alpha P_\alpha f_\alpha\nonumber\\
		&=&c_R^{-1}\bigg(\int_{F}f(w)K_w(z)dv(w)-\int_{F\backslash D(z,R)}f(w)K_w(z)dv(w)-\int_{F}f(w)II(z,w)dv(w)\bigg)\nonumber\\
		&=&c_R^{-1}\bigg(\int_{F}f(w)I(z,w)dv(w)-\int_{F\backslash D(z,R)}f(w)K_w(z)dv(w)\bigg).
		\end{eqnarray}
		
		Each term of $I$ has the form $\frac{(a(z')-z'')^\alpha\overline{w''^\alpha}}{(1-\la p(z),w\ra)^{n+1+j}}$, where $0\leq j\leq m-1$ and $|\alpha|=j$. We want to prove that their integrals with $f$ are small. First, since $E_3\subset F$, $E_3\subset\{z\in\bn: 1-|z|^2>2^{-6s}\}$, which is contained in a compact subset of $\bn$, we can find $C_1>0$, depending only on $ m, R$, such that 
		\[
	\bigg|\frac{(a(z')-z'')^\alpha\overline{w''^\alpha}}{(1-\la p(z),w\ra)^{n+1+j}}-\frac{(-z'')^\alpha\overline{w''^\alpha}}{(1-\la \pi(z),w\ra)^{n+1+j}}\bigg|\leq C_1|a(z')|\leq C_1\delta,\quad\forall z\in E_3, w\in F.
		\]
	Therefore
	\begin{equation}\label{eqn: I}
	\big|\int_{F}f(w)I(z,w)dv(w)-\int_{F}f(w)\sum_{j=0}^{m-1}c_j\frac{\la-z'',w''\ra^j}{(1-\la z',w'\ra)^{n+1+j}}dv(w)\big|\lesssim C_1\delta\int_{F}|f(w)|dv(w).
	\end{equation}
	Notice that for each $w'\in U'$, we have $|1-\la w',u'\ra|>\frac{1}{2}s^{-2}$, and therefore
	\[\pi^{-1}(w')\cap F=\{(w',w''): 1-|w'|^2-2^{-s}|1-\la w',u'\ra|^2<|w''|^2<1-|w'|^2-2^{-6s}\}.\]
	Each fiber is either a spherical shell or a ball in $\CC^{n-d}$. From this and Lemma \ref{lem: sphereical shell} (2), it is easy to see that for $z'\in \pi(E_3)$,
	\begin{eqnarray}\label{eqn: integral for I}
	&&\bigg|\int_{F}f(w)\sum_{j=0}^{m-1}c_j\frac{\la-z'',w''\ra^j}{(1-\la z',w'\ra)^{n+1+j}}dv(w)\bigg|\\
	&\lesssim&\sum_{j=0}^{m-1}\int_{U'}\frac{1}{|1-\la z',w'\ra|^{n+1+j}}dv(w')\sup_{z'\in\pi(E_3), |\alpha|\leq m-1}|\partial^\alpha f(z',0)|\\
	&\lesssim&2^{6s(n+m)}\sup_{z'\in \pi(E_3), |\alpha|\leq m-1}|\partial^\alpha f(z',0)|.
	\end{eqnarray}
	It remains to show that $\sup_{z'\in\pi(E_3), |\alpha|\leq m-1}|\partial^\alpha f(z',0)|$ is sufficiently small. Obtain $F_1$ by filling the ``holes'' in $F$, i.e.,
	\[
	F_1=\{(w',w'')\in\bn: w'\in U', |w''|^2<1-|w'|^2-2^{-6s}\}.
	\]
By Lemma \ref{lem: E_u properties}, $E_3\subset D(0,4s\log2)$ and $F\supset E_4$. Thus for $\delta$ small enough, the set of points $\{\pi(z): z\in E_3\}\cup\{p(z): z\in E_3\}$ has a positive Euclidean distance (independent of $k$) to the boundary of $F_1$. So there exists a constant $C_2>0$ (depending only on $d, m, R$) such that $\forall z\in E_3, \forall g\in\mathrm{Hol}(F_1)$ and $\forall \alpha$, $|\alpha|\leq m-1$,
\[
|(\partial^\alpha g)(p(z))-(\partial^\alpha g)(\pi(z))|^2\leq C_2^2|p(z)-\pi(z)|^2\int_{F_1}|g(w)|^2dv(w)\leq C_2^2\delta^2\int_{F_1}|g(w)|^2dv(w).
\]
Since $f|_{mA}=0$, we have $(\partial^\alpha f)(p(z))=0$ for all $z\in E_3$ and $\alpha=(\alpha_{d+1},\cdots,\alpha_n)$, $|\alpha|\leq m-1$. Thus
\begin{equation}\label{eqn: f pi(z) small}
\sup_{z'\in\pi(E_3),|\alpha|\leq m-1}|(\partial^\alpha f)(z',0)|^2\leq C_2^2\delta^2\int_{F_1}|f(w)|^2dv(w),\quad\forall z\in E_3.
\end{equation}
We need to compare the $L^2$-norms on $F_1$ and $F$. For any $w'\in U'$, if $1-|w'|^2-2^{-s}|1-\la w',u'\ra|^2>0$, then for $s$ large enough, we have
\begin{eqnarray*}
	&&\frac{1-|w'|^2-2^{-6s}}{(1-|w'|^2-2^{-6s})-(1-|w'|^2-2^{-2s}|1-\la w',u'\ra|^2)}\\
	&=&\frac{1-|w'|^2-2^{-6s}}{2^{-2s}|1-\la w',u'\ra|^2-2^{-6s}}\\
	&\leq&\frac{1}{2^{-2s-2}s^{-4}-2^{-6s}}\leq 2^{2s+3}s^4.
	\end{eqnarray*}
Then by Lemma \ref{lem: sphereical shell} and a simple double integral argument, we have
\begin{equation}\label{eqn: norm on shell and ball}
\int_{F_1}|g|^2dv\leq 2^{2s+3}s^4\int_{F}|g|^2dv,\quad\forall g\in\mathrm{Hol}(F_1).
\end{equation}
	
Combining inequalities \eqref{eqn: integral for I}\eqref{eqn: f pi(z) small}\eqref{eqn: norm on shell and ball}, we get
\[
|\int_{F}f(w)\sum_{j=0}^{m-1}c_j\frac{\la-z'',w''\ra^j}{(1-\la z',w'\ra)^{n+1+j}}dv(w)|\lesssim2^{6s(n+m)}C_2\delta\cdot 2^{s+2}s^2\bigg(\int_{F}|f|^2dv\bigg)^{1/2}.
\]	
Then by inequality \eqref{eqn: I} and Holder's inequality, we get
\[
\big|\int_{F}f(w)I(z,w)dv(w)\big|\lesssim (C_1\delta+2^{6s(n+m)}C_2\delta\cdot 2^{s+2}s^2)\bigg(\int_{F}|f|^2dv\bigg)^{1/2}.
\]		
If we choose $\delta$ small enough, we can make $(C_1\delta+2^{6s(n+m)}C_2\delta\cdot 2^{s+2}s^2)\leq\epsilon_R$.	Then inequality (1) follows from inequality \eqref{eqn: (1)}, Lemma \ref{lem: norm estimates for integral kernels} and our estimates above. This completes the proof.	
	\end{proof}
	
	The following Lemma is a simplified version of Lemma \ref{lem: decomposition on Fu} (1).
	\begin{lem}\label{lem: u decomposition}
		Suppose $J$ is a prime ideal in $\poly$. Write $A=Z(J)$. Suppose $A$ has no singular point in $\partial\bn$ and intersects $\partial\bn$ transversely. Then for  a positive integer $m$ there exist a positive integer $N$ and constant $C>0$ with the following property. For $R>0$ sufficiently large, there exists a positive integer $K$ such that the following hold. Let $s'$ be the positive integer such that $(s'-1)\log2\leq16R<s'\log2$ and $s=10s'$,
		\begin{itemize}
			\item[(1)]
		For each $u\in A$, $1-|u|^2=2^{-2sk}$, $k\geq K$, define $E_u$ and $F_u$ as in the beginning of Section \ref{section: a covering lemma}. Then there is a finite set of polynomials $\{p_l\}\subset J^m$, $\sup_l\deg p_l\leq N$, and linear maps $J\to\mathrm{Hol}(E_{u3})$, $f\mapsto f_l$ satisfying the following inequalities.
		\begin{itemize}
			\item[(i)]$\int_{E_{u3}}\big|\sum_l p_lf_l-f\big|^2dv\leq C\epsilon_{R}^2\int_{F_u}|f|^2dv$.
			\item[(ii)]$\int_{E_{u3}}\big(\sum_l|p_lf_l|\big)^2dv\leq C\int_{F_u}|f|^2dv$.
		\end{itemize}
	\item[(2)] For each $v\in\bn$, $\beta(v, A)>5s'\log2$, $1-|v|^2=2^{-2s'k}, $ $k\geq10K$, define $F_v'$ as in the beginning of Section \ref{section: a covering lemma}. Then there exists a polynomial $p\in J^m$ such that $p$ is non-vanishing on $F_v'$ and $\deg p\leq N$.
\end{itemize}
	\end{lem}

\begin{proof}
	We want to apply Lemma \ref{lem: general decomposition lemma}. A M\"{o}bius transform $\varphi_u$ will allow us to transfer between decompositions on $E_{u3}$ and $\E_{u3}$. Also, we need to make adjustments so that the generating functions $p_l$ are polynomials in $J^m$.
	
	Denote $d=\dim A$. For any $\zeta\in\partial\bn\cap A$, since $A$ intersects $\partial\bn$ transversely at $\zeta$, $A$ is not contained in $\zeta^\perp+\zeta$. By Theorem \ref{thm: proper projections on tangent plane}, the set 
	\[
	\tilde{\mathcal{U}}_\zeta=\{l\in U(\zeta^\perp):  \overline{(l\oplus 1)(NT_\zeta A)}^{\perp}\cap \overline{A-\zeta}=\emptyset\}.
	\]
	is dense in $U(\zeta^\perp)$. Here we denote $L^\perp$ the orthogonal complement of a linear space $L$ in the complex projective space $\mathbb{P}_n$ (cf. Appendix) and $U(\zeta^\perp)$ the space of all unitary transforms on $\zeta^\perp\cong\CC^{n-1}$. We consider $l\oplus 1$ as acting on the $\zeta^\perp$ by $l$ and acting on $\CC\zeta$ as identity.
	
	Let $\delta>0$ be the constant in Lemma \ref{lem: general decomposition lemma} determined by $d, m, R$. If $\overline{NT_\zeta A}^\perp\cap \overline{A-\zeta}=\emptyset$, then take $L_\zeta=NT_\zeta A$. Otherwise, choose $l\in \tilde{\mathcal{U}}_\zeta$ close enough to the identity $I_{\zeta^\perp}$ so that the Hausdorff distance 
	\[
	\dist_H(L_\zeta\cap\bn, NT_\zeta A\cap\bn)\leq\delta/4,
	\]
	where $L_\zeta=(l\oplus1)(NT_\zeta A)$. In either case we have $\zeta\in L_\zeta$, $\overline{L_\zeta}^\perp\cap\overline{A-\zeta}=\emptyset$, and 
	$$
	\dist_H(L_\zeta\cap\bn, NT_\zeta A\cap\bn)\leq\delta/4.
	$$
	Choose a new basis $\{e_1,\cdots,e_n\}$ such that $L_\zeta=\mathrm{span}\{e_1,\cdots,e_d\}$. Denote $\pi_\zeta$ the projection from $\cn$ onto $L_\zeta$. 
	
	By Lemma \ref{lem: nonintersection implies proper}, $\pi_\zeta|_{A-\zeta}$ is proper. Since $\zeta$ is a regular point of $A$ and $A$ intersects $\partial\bn$ transversely,  $\pi_\zeta$ defines a one-sheeted analytic cover of $A-\zeta$ in a small neighborhood of $0$. Apply Theorem \ref{thm: local CDF} to $A-\zeta$ and $L_\zeta$. Let $U_1$, $\mathcal{U}$ and $\W$ be as in Theorem \ref{thm: local CDF}. Thus for any $w\in \W$ and $l\in\mathcal{U}$, the function
	\[
	\frac{P_{\pi_\zeta|_{l(A-\zeta)}}(z,w)}{P_{\pi_\zeta|_{l(A-\zeta)\cap U_1}}(z,w)}
	\]
	is non-vanishing and holomorphic in $U_1$.
	
	Start with an open neighborhood $U$ of $\zeta$. For $u\in A\cap U$, let $L_u=L_\zeta\cap u^\perp+\CC u$ and denote $\pi_u$ the projection from $\cn$ onto it. Since $\zeta\in L_\zeta$, the definition is consistent at $\zeta$, and the spaces $L_u$ vary continuously with $u$. Thus for $u$ close enough to $\zeta$, we can find $l_u\in\mathcal{U}$ such that $l_u^{-1}(L_\zeta)=L_u$. So $\pi_u=l_u^{-1}\pi_\zeta l_u$. For any $w\in \W$,
	\[
	\frac{P_{\pi_\zeta|_{l_u(A-\zeta)}}(z,w)}{P_{\pi_\zeta|_{l_u(A-\zeta)\cap U_1}}(z,w)}
	\]
	is non-vanishing in $U_1$. Equivalently, for any $w\in \W$,
	\begin{equation}\label{eqn: 6.5.1}
	\frac{P_{\pi_\zeta|_{l_u(A)}}(z,w)}{P_{\pi_\zeta|_{l_u(A)\cap (U_1+l_u(\zeta))}}(z,w)}
	\end{equation}
	is non-vanishing in $U_1+l_u(\zeta)$. Note that the definition of canonical defining functions depend on the choice of a basis. Let $e_{u,i}=l_u^{-1}(e_i)$. By Remark \ref{rem: CDF basis}, under the new basis $\{e_{u,1},\cdots,e_{u,n}\}$, we have
	\[
	\psi_w(z):=\frac{P_{\pi_u|_A}(z,w)}{P_{\pi_u|_{A\cap(l_u^{-1}U_1+\zeta)}}(z,w)}
	\]
	is non-vanishing for $z\in l_u^{-1}U_1+\zeta$ and $w\in \W$. It is easy to see that $l_u^{-1}(U_1)+\zeta\supset U_2, \forall u\in U$, for some open neighborhood $U_2$ of $\zeta$. Let us use $\PPP$ with appropriate subscripts to denote the locally defined canonical defining functions depending on a one-sheeted analytic cover on the piece of manifold $A\cap U$, or its image under a M\"{o}bius transform. Then we have shown that for any $u\in U$ and $w\in \W$,
	\[
	\psi_w(z)=\frac{P_{\pi_u|_A}(z,w)}{\PPP_{\pi_u|_A}(z,w)}
	\]
is non-vanishing on $U_2$.

	On the other hand, both $L_u$ and the normal tangent spaces $NT_u A$ vary continuously. By shrinking the neighborhood $U$ we can also assume that 
	\[
	\dist_H(L_u\cap\bn, NT_u A\cap\bn)\leq\delta/2,\quad\forall u\in U\cap A.
	\]
	By Lemma \ref{lem: Z close to NT} and \ref{lem: E_u properties} (2),again, by shrinking $U$, we also have for any $u\in U$,
	\[
	\dist(z, NT_uA)\leq\delta/2, \quad\forall z\in \varphi_u(A)\cap\F_u.
	\]
	Thus
	\[
	\dist(z, L_u)\leq\delta, \quad \forall z\in \varphi_u(A)\cap\F_u.
	\]
    
    For any $f\in J^m$, define $U_uf(z)=f\circ\varphi_u(z)\cdot k_u(z)$, where $k_u(z)=\frac{(1-|u|^2)^{(n+1)/2}}{(1-\la z,u\ra)^{n+1}}$. Then $U_uf|_{m\varphi_u(A)}=0$. Applying Lemma \ref{lem: general decomposition lemma} to $F=\F_u$ and $\varphi_u(A)$, we can find $F_\alpha$ such that
    \[
    \int_{\E_{u3}}|U_uf-\sum_{|\alpha|=m_j}\PPP_\alpha F_\alpha|^2dv\leq C\epsilon_{R}^2\int_{\F_{u}}|U_uf|^2dv
    \]
    and
    \[
    \int_{\E_{u3}}\bigg(\sum_{|\alpha|=m_j}|\PPP_\alpha F_\alpha|\bigg)^2dv\leq C\int_{\F_{u}}|U_uf|^2dv.
    \]
    Here $\PPP_\alpha=\PPP_{\pi_u|_{m\varphi_u(A)}, \alpha}$, i.e., $\sum_{|\alpha|=m}\PPP_\alpha(z)\overline{w^\alpha}=\PPP_{\pi_u|_{\varphi_u(A)}}^m(z,w)$.
    Therefore
    \begin{eqnarray*}
    	&&\int_{E_{u3}}|f(z)-\sum_{|\alpha|=m}\PPP_\alpha\circ\varphi_u(z)F_\alpha\circ\varphi_u(z)k_u(z)|^2dv(z)\\
    	&=&\int_{\E_{u3}}|U_uf-\sum_{|\alpha|=m}\PPP_\alpha F_\alpha|^2dv(z)\\
    	&\leq&C\epsilon_{R}^2\int_{\F_{u}}|U_uf|^2dv(z)\\
    	&=&C\epsilon_{R}^2\int_{F_{u}}|f(z)|^2dv(z).
    \end{eqnarray*}
    Similarly, for each pair $\alpha, \beta, |\alpha|=|\beta|=m$,
    \[
    \int_{E_{u3}}|\PPP_\beta\circ\varphi_u(z) F_\alpha\circ\varphi_u(z)k_u(z)|^2dv(z)\leq C\int_{F_{u}}|f(z)|^2dv(z).
    \]
    
   Using the formula for $\varphi_u$, it is easy to verify that
	\begin{equation}\label{eqn: 6.5.2}
	\PPP_{\pi_u|_{\varphi_u(A)}}(\varphi_u(z), w)=-\frac{(1-|u|^2)^{1/2}}{1-\la z,u\ra}\PPP_{\pi_u|_{A}}(z,w).
	\end{equation}
	
	Let $\Gamma=\{\alpha=(\alpha_{d+1},\cdots,\alpha_n): |\alpha|=m\}$ and let $K=\#\Gamma$. By Lemma \ref{lem: w alpha linearly independent}, we can choose $K$ distinct points $\{w_i\}_{i=1}^K\subset \W$ (independent of $u$) such that the vectors $\{(w_i^\alpha)_{\alpha\in\Gamma}\}_{i=1}^K$ form a basis of $\CC^K$. Let $W$ be the $K\times K$ matrix $\big[\overline{w_i^\alpha}\big]_{\alpha\in\Gamma, i=1,\cdots,K}$. Then $W$ is invertible.
	Suppose $W^{-1}=(c_{i,\alpha})$.
	
	Denote $p_i(z)=P_{\pi_u|_{A}}^m(z,w_i)$ and  $\psi_i=\psi_{w_i}$, $i=1,\cdots,K$. Since $J$ is prime,
	by Theorem \ref{thm: global CDF are polynomials} and Hilbert's Nullstellensatz, each $p_i$ is a polynomial in $J^m$. Denote $N_0$ the upper bound of the degrees of the canonical defining functions of $A$, as in Theorem \ref{thm: global CDF are polynomials}. Then we have $\deg p_i\leq mN_0$. Set $N=mN_0$.
	Let 
	\[
	f_i(z)=\sum_\alpha \big(-\frac{(1-|u|^2)^{1/2}}{1-\la z,u\ra}\big)^{m}c_{i,\alpha}\psi_i(z)^{-m}F_\alpha\circ\varphi_u(z)k_u(z),
	\]
	then 
	\begin{eqnarray*}
		p_i(z)f_i(z)&=&\sum_\alpha \big(-\frac{(1-|u|^2)^{1/2}}{1-\la z,u\ra}\big)^{m}c_{i,\alpha}\PPP_{\pi_u|_A}(z,w_i)^{m}F_\alpha \circ\varphi_u(z)k_u(z)\\
		&=&\sum_{\alpha}c_{i,\alpha}\PPP_{\pi_u|_{\varphi_u(A)}}^m(\varphi_u(z),w_i)F_\alpha \circ\varphi_u(z)k_u(z)\\
		&=&\sum_\alpha\sum_\beta c_{i,\alpha}\overline{w_i^\beta}\PPP_\beta\circ\varphi_u(z)F_\alpha\circ\varphi_u(z)k_u(z).
	\end{eqnarray*}
	and
	\begin{eqnarray*}
		\sum_i p_i(z)f_i(z)&=&\sum_\alpha\sum_\beta\sum_i c_{i,\alpha}\overline{w_i^\beta}\PPP_\beta\circ\varphi_u(z)F_\alpha\circ\varphi_u(z)k_u(z)\\
		&=&\sum_\alpha\sum_\beta\delta_{\alpha,\beta} \PPP_\beta\circ\varphi_u(z)F_\alpha\circ\varphi_u(z)k_u(z)\\
		&=&\sum_\alpha \PPP_\alpha\circ\varphi_u(z)F_\alpha\circ\varphi_u(z)k_u(z).
	\end{eqnarray*}
	
		From this it is easy to see that the functions $p_i$ and $f_i$ satisfy inequalities (i) and (ii). 
	
	For each $\zeta\in A\cap\partial\bn$ we have found a neighborhood $U_\zeta:=U$ such that for any $u\in U_\zeta$, we have a decomposition with the stated properties. By compactness, we can cover $A\cap\partial\bn$ by finitely many such neighborhoods. This proves (1).

	Next we prove (2). Suppose $v\in\bn$ and $\beta(v, A)>5s'\log2$. By Lemma \ref{lem: E_u properties} (2), $F_v'=E_{v3}'\subset D(v,4s'\log2)$. Let $\xi\in A$ be a point such that $\beta(v,\xi)=\beta(v,A)$. Since
	$$\beta(\varphi_\xi(v), 0)=\beta(v,\xi)=\beta(v,A)>5s'\log2,$$
	by Lemma \ref{lem: beta rho} (2),
	\[
	1-|\varphi_\xi(v)|^2\leq 4e^{-2\beta(\xi,v)}<2^{-10s'+2}.
	\]
	By \cite[2.2.7]{Rudin}, the Euclidean diameter of $D(\varphi_\xi(v), 4s'\log2)$ do not exceed 
	$2^{-s'}$. Since $\varphi_\xi(F_v')\subset\varphi_\xi(D(v,4s'\log2))=D(\varphi_\xi(v),4s'\log2)$, the Euclidean diameter of $\varphi_\xi(F_v')$ do not exceed $2^{-s'}$.

	Also, we can assume that $v$ is close enough to $\partial\bn$ (equivalently, $k$ is large enough) so that $\xi$ is a regular point of $A$. A simple computation (cf. \cite{DWY2} Lemma 3.12) shows that $\varphi_\xi(v)$ is perpendicular to $T_0\varphi_\xi(A)=\varphi_\xi(T_\xi A+\xi)$ (cf. \cite[Proposition 2.4.2]{Rudin}).

	 Let $\delta>0$ be a sufficiently small constant to be determined later. Suppose $\zeta\in A\cap\partial\bn$. In the proof of (1), we have constructed an open neighborhood $U$ of $\zeta$ and an open subset $\W$ in $\CC^{n-d}$ such that for any $u\in U\cap A$ and any $w\in \W$,
	\[
	\psi_w(z)=\frac{P_{\pi_u|_A}(z,w)}{\PPP_{\pi_u|_A}(z,w)}
	\]
	is non-vanishing on $U_2$, where $U_2$ is another open neighborhood of $\zeta$. Here the canonical defining functions are constructed based on the basis $\{e_{u,1},\cdots, e_{u,n}\}$ as in the proof of (1). For $v$ close enough to $\zeta$ we can assume that $u\in U$ and $F_v'\subset U_2$. Thus by shrinking $U$ we will have that $\psi_w(z)$ is non-vanishing on $F_v'$ for any $v\in U$.

	Notice that although $\W$ depends on our choice of $\delta$, from the proof of Theorem \ref{thm: local CDF} (which is used in the proof of (1)), by shrinking the set $U$, we can always ensure that the volume  of $\Pi(\W)$ is greater than half of the volume of $\mathbb{P}_{n-d-1}$. Here $\Pi$ denotes the canonical map from $\CC^{n-d}\backslash\{0\}$ to $\mathbb{P}_{n-d-1}$. 
	
	On the other hand, from the previous argument, we know that the Euclidean diameter of $\varphi_\xi(F_v')$ is less than $2^{-s'}$, $\varphi_\xi(v)$ is perpendicular to $\varphi_\xi(T_\xi A+\xi)$, and that $1-|\varphi_\xi(v)|^2\leq 4e^{-2\beta(\xi,v)}<2^{-10s'+2}$. By Lemmas \ref{lem: NT close to T}, \ref{lem: Z close to NT} and our construction in (1), for $v$ close enough to $\partial\bn$, we will have 
	\[
	\dist_H(\varphi_\xi(A)\cap D(0, 5s\log2), L_\xi\cap D(0,5s\log2))\leq 2\delta.
	\]
	Denote $\W_1=\{w\in\CC^{n-d}: \PPP_{\pi_\xi|_{\varphi_\xi(A)}}(z,w)\text{ is non-vanishing on }\varphi_\xi(F_v')\}$.
	For $s'$ large enough and $\delta$ small enough, we will have that the volume of $\Pi(\W_1)$ is greater than half of the volume of $\mathbb{P}_{n-d-1}$. Choose such a $\delta$. Then $\W\cap\W_1$ is non-empty. Choose a $w_0\in\W\cap\W_1$. 
	Let $U$ be the open neighborhood of $\zeta$ determined by $\delta$. Then for any $v\in U$, by \eqref{eqn: 6.5.2} and the fact that $w_0\in\W_1$, $\PPP_{\pi_\xi|_A}(z,w_0)$ is non-vanishing on $F_v'$. Since $w_0\in\W$, $\frac{P_{\pi_\xi|_A}(z,w_0)}{\PPP_{\pi_\xi|_A}(z,w_0)}$ is also non-vanishing on $F_v'$. Thus $P_{\pi_\xi|A}(z,w_0)$ is non-vanishing on $F_v'$. By Theorem \ref{thm: global CDF are polynomials} and Hilbert's Nullstellensatz, $P_{\pi_\xi|_A}(z,w_0)$ is a polynomial of degree less than $N$ in $J$. Thus $p=P_{\pi_\xi|_A}^m(z,w_0)\in J^m$ is non-vanishing on $F_v'$ and $\deg p\leq N$. This proves (2).
	
	\end{proof}
	
	We are ready to prove Lemma \ref{lem: decomposition on Fu}.
	
	\begin{proof}[\textbf{Proof of Lemma \ref{lem: decomposition on Fu}}]
		Notice that the varieties $Z_j$ are disjoint in a neighborhood of $\partial\bn$. If $\zeta\in Z_j\cap\partial\bn$, then for each $i\neq j$, there exists $p_i\in I_i^{m_i}$ such that $p_i(\zeta)\neq0$. Thus for each $\zeta\in Z_j\cap\partial\bn$, we can find a neighborhood $V_\zeta$ and $k-1$ polynomials $p_i\in I_i^{m_i}$, $i\neq j$ such that $p_i$ are non-vanishing on $V_\zeta$. Choose finitely many open sets $V_\zeta$ that cover $Z\cap\partial\bn$. Suppose $F_u\subset V_\zeta$, where $\zeta\in Z_j\cap\partial\bn$. For $f\in I$, let $P_l$, $f_l$ be the functions constructed in Lemma \ref{lem: u decomposition}, for $I_j^{m_j}$. Then we can simply replace $P_l$ with $p_l\Pi_{i\neq j}p_i$ and $f_l$ with $\frac{f_l}{\Pi_{i\neq j}p_j}$. This proves (1). The proof for (2) is similar. 
	\end{proof}
	
	We are ready to prove Theorem \ref{thm: Smooth implies ASD}.
	\begin{proof}[\textbf{Proof of Theorem \ref{thm: Smooth implies ASD}}]
		For any $\epsilon>0$ sufficiently small, let $C$ be the constant in Lemma \ref{lem: decomposition on Fu}. Choose $R_0>0$ so that $C\epsilon_{R_0-3}^2<\epsilon$. Let $\{E_u\}_{u\in\LL}$, $\{E_v'\}_{v\in\LL'}$, $\{F_u\}_{u\in\LL}$, $\{F_v'\}_{v\in\LL'}$ be determined by $R$ as in Section \ref{section: a covering lemma}. Let $K$, $N$ be the positive integers in Lemma \ref{lem: decomposition on Fu}. Let $\delta=\sqrt{1-2^{-2m(K+1)}}$. Finally, let $\{E_i\}=\{E_u\}_{u\in\cup_{k\geq K}\LL_k}\cup\{E_v'\}_{v\in\cup_{k\geq10K}\LL_k'}$. The polynomials $p_{ij}$ will be the corresponding polynomials in Lemma \ref{lem: decomposition on Fu}. The conditions (1) and (2) in Hypothesis 1 follow from Lemma \ref{lem: E_u properties}, where we take $M=2^{10n+3}.$ The conditions (3) and (4)(i)(ii)(iv) follow from Lemma \ref{lem: decomposition on Fu}. By Lemma \ref{lem: E_u properties} (2), $(1-|\lambda|^2)$ is comparable to $(1-|u|^2)$ for $\lambda\in F_u$ and  $(1-|\lambda|^2)$ is comparable to $(1-|v|^2)$ for $\lambda\in F_v'$. From this, condition (4)(iii) follows immediately.
		
		Therefore $I$ satisfies Hypothesis 1. The rest of the theorem follow from Theorem \ref{thm: ASD}. This completes the proof.
	\end{proof}

	\section{Concluding Remarks}
	In this paper, we have provided a unified proof for most known results of the Arveson-Douglas Conjecture. In fact, we have proved the stronger result that the submodules under our consideration have the asymptotic stable division property. We raise the following question.
	
	\medskip
	\noindent\textbf{Question}: Suppose $I$ is an ideal in $\poly$. Find sufficient conditions for $I$ to have the asymptotic stable division property. Find sufficient conditions for $I$ to have the asymptotic stable division property with generating elements $\{h_i\}$ being polynomials of uniformly bounded degrees. 
	
	\medskip
	By Theorem \ref{thm: ASD to EN}, a positive result on this question will lead to a positive result of the Arveson-Douglas Conjecture. 
	We would also like to explore other applications of the asymptotic stable division property, for example, in index theory.

	\medskip
	The techniques we have developed in this paper are aimed at getting more general results. 
	For the next step, we plan to consider the following examples.
	\begin{itemize}
		\item[(1)] Arbitrary union of smooth, transversal varieties.
		\item[(2)] Varieties with certain type of singular points on $\partial\bn$. For example, singular points $\zeta$ with tangent cones being linear subspaces. This will cover the classic example of singular point, $z_1^2=z_2^3$ at $0$.
	\end{itemize}
Tools, for example, from \cite{Chi} and \cite{LT04} can be useful.
	
	We will also use the techniques to study the Arveson-Douglas Conjecture in connection with the $L^2$-extension problem \cite{Zho}. The covering constructed in Section \ref{section: a covering lemma} can be useful in constructing a holomorphic extension.

	\section{Appendix}\label{appendix}
	
	For an algebraic set $A\subset\cn$ with pure dimension, we can show that the functions $P_\pi$ and $P_{\pi,\alpha}$ in Definition \ref{defn: CDF} are polynomials in $z$ and $\bar{w}:=(\bar{w}_{p+1},\cdots,\bar{w}_n)$. Let us first consider a simple case.
	\begin{lem}\label{lem: global CDF, codimension 1}
		Suppose $A\subset\cn$ is an algebraic set of pure dimension $n-1$, and suppose 
		\[
		\pi: A\to \CC^{n-1}, z\mapsto z':=(z_1,\cdots, z_{n-1})
		\]
		is proper. Then $P_\pi(z,w)$ and $P_{\pi,\alpha}(z)$ are polynomials in $z$ and $\bar{w}$. 
	\end{lem}
	\begin{proof}
		In this case, there is only one canonical defining function. If $\pi$ is $k$-sheeted, then $\alpha=k$, and 
		\[
		P_\pi(z,w_n)=P_{\pi,k}(z)\bar{w}_n^k.
		\]
		It suffices to show that $P_{\pi,k}$ is a polynomial. By definition, $P_{\pi,k}$ is the Weierstrass polynomial determined by $\pi$.
		
		The algebraic set $A$ decomposes into finitely many irreducible algebraic sets. It is easy to see that $P_{\pi,k}$ is just the product of the canonical defining functions of the irreducible components. Without loss of generality, we can assume $A$ is itself irreducible. 
		
		By \cite[Proposition 1.13]{Har}, there is an irreducible polynomial $Q(z)$ such that $A=Z(Q)$. Clearly $\deg_{z_n}Q\geq k$. Write $l=\deg_{z_n}Q$ and write $Q$ as a polynomial in $z_n$ with coefficients in $\CC[z_1,\cdots,z_{n-1}]$,
		\[
		Q(z)=q_l(z')z_n^l+\cdots.
		\]
		Consider the set
		\[
		B=\{z\in A: q_l(z')= 0,\text{ or } \partial_nQ(z)=0\}.
		\]
		$\pi(B)$ is the set of points $z'\in\CC^{n-1}$ such that $Q(z',\cdot)$ do not have $l$ distinct simple roots. Since $Q$ is irreducible, $B$
		is an analytic subset of dimension $\leq n-2$. By \cite[Proposition 3.3.2]{Chi}, $\pi(B)\subset\CC^{n-1}$ is an analytic subset of dimension $\leq n-2$. Thus $\CC^{n-p}\backslash\pi(B)$ is dense in $\CC^{n-1}$. If $l>k$, for any $z'\in\CC^{n-1}\backslash\pi(B)$, $\pi^{-1}(z')$ contains $l>k$ distinct points, a contradiction. Thus $\deg_{z_n}Q=k$.
		Also, if $\deg q_k>0$, then for $z'\notin\sigma$, comparing the two polynomials in $z_n$,
		\[
		Q(z',z_n)=q_k(z')z_n^k+\cdots
		\] 
		and
		\[
		P_{\pi,k}(z',z_n)=(z_n-a_1(z'))\cdots(z_n-a_k(z'))
		\]
		we get $Q(z',z_n)=q_k(z')P_{\pi,k}(z',z_n)$, $z'\notin\sigma, z_n\in\CC$. Since $\sigma$ is nowhere dense, $Q\equiv q_kP_{\pi,k}$. But then if $q_k(z')=0$ at some $z'$, we will have infinitely many points on the fiber $\pi^{-1}(z')$, a contradiction. Thus $q_k$ is a constant. Without loss of generality, we assume $q_k\equiv1$. Then $P_{\pi,k}=Q$ is a polynomial. This completes the proof.
	\end{proof}

\begin{lem}\label{lem: w alpha linearly independent}
	For any open set $U\subset\bn$ and any finite collection of indexes 
	\[
	F\subset\{\alpha=(\alpha_1,\cdots,\alpha_n): \alpha_i\in\mathbb{N}\},
	\]
	let $K=\# F$. Then there exists $K$ distinct points $\{w_i\}_{i=1}^k$ in $U$ such that the vectors
	$\{(w_i^\alpha)_{\alpha\in F}\}_{i=1}^K$ in $\CC^K$ are linearly independent.
\end{lem}
\begin{proof}
	Let 
	\[
	L=\mathrm{span}\{(w^\alpha)_{\alpha\in F}: w\in U\}.
	\]
     It suffices to show that $L=\CC^K$. Otherwise, choose a non-zero vector $(a_\alpha)_{\alpha\in F}$ that is perpendicular with $L$. Define
     \[
     f(w)=\sum_{\alpha\in F}\overline{a_\alpha}w^\alpha.
     \]
     Then $f$ is a analytic polynomial that vanishes on the open set $U$. Therefore $f$ is identically zero. So $a_\alpha=0, \forall \alpha\in F$, a contradiction. This completes the proof.
	\end{proof}
	
	\begin{thm}\label{thm: global CDF are polynomials}
		Suppose $A\subset\cn$ is an algebraic set of pure dimension $p$. Suppose the projection
		\[
		\pi: A\to\CC^p, z\mapsto z':=(z_1,\cdots,z_p)
		\]
		is proper. Then the functions  $P_\pi(z,w)$ and $P_{\pi,\alpha}(z)$ are polynomials in $z$ and $\bar{w}$. Moreover, there exists a positive integer $N$, depending only  on $A$, such that for any choice of basis and any proper projection $\pi$, the degrees of $P_{\pi}(z,w)$ (in $z$) and $P_{\pi,\alpha}$ are less than $N$.
	\end{thm}
	\begin{proof}
		Let $\sigma\subset\CC^p$ be the critical set of $\pi$. Suppose $\pi$ is $k$-sheeted. Fix any $z'\notin\sigma$, we have $\pi^{-1}(z')=\{(z',a_1(z')),\cdots,(z',a_k(z'))\}$, where
		$\{a_1(z'),\cdots,a_k(z')\}$ are $k$ distinct points in $\CC^{n-p}$. The set 
		\[
		\W:=\{w\in\CC^{n-p}:  \la a_1(z'), w\ra,\cdots,\la a_k(z'),w\ra \text{ are distinct at least for one }z'\in U'\},
		\]
		is open in $\CC^{n-p}$. Fix any $w\in \W$. Consider the projections
		\[
		\pi_w:A\to\CC^{p+1}, z\mapsto(z',\la z'',w\ra)
		\] 
		and 
		\[
		\pi_w':\pi_w(A)\subset\CC^{p+1}\to\CC^p, (z',\la z'',w\ra)\to z'.
		\]
		Then $\pi=\pi_w'\circ\pi_w$. By \cite[3.1 (2)]{Chi}, both $\pi_w$ and $\pi_w'$ are proper maps. By \cite[Theorem 3.2, Proposition 3.3.2]{Chi}, $\pi_w(A)$ is a pure algebraic set in $\CC^{p+1}$ of dimension $p$, and $\pi_w'$ is a $k$-sheeted analytic cover. By Lemma \ref{lem: global CDF, codimension 1}, $P_{\pi_w'}((z',\lambda),\xi)$ is a polynomial in $z'$, $\lambda$ and $\bar{\xi}$. Also, from the proofs of \cite[Theorem 3.2, Proposition 3.3.2]{Chi}, the degree (in $(z',\lambda)$) of $P_{\pi_w'}((z',\lambda),\xi)$ has a upper bound determined by any set of generators of the ideal $\{p\in\poly: p|_A=0\}$, which we denote by $N$. Checking by definition, we have the equation
		\[
		P_\pi(z,w)=P_{\pi_w'}((z',\la z'',w\ra),1).
		\]
		Thus for any $w\in \W$, $P_\pi(z,w)$ is a polynomial in $z$ and $\deg_z P_{\pi}(z,w)\leq N$. Fix an order of the set $\Gamma=\{\alpha=(\alpha_{p+1},\cdots,\alpha_n): |\alpha|=k\}$, and let $K=\#\Gamma$. By Lemma \ref{lem: w alpha linearly independent}, we can choose $K$ distinct points $\{w_j\}_{j=1}^K\subset \W$ such that the $K\times K$ matrix $W=(\overline{w_j^{\alpha}})_{\alpha\in\Gamma, j=1,\cdots, K}$ is invertible. From the equations
		\[
		P_\pi(z,w_j)=\sum_{\alpha\in\Gamma}P_{\pi,\alpha}(z)\overline{w_j^\alpha},\quad j=1,\cdots,K,
		\]
		we can solve $P_{\pi,\alpha}$ as linear combinations of $\{P_\pi(z,w_j)\}_{j=1}^K$. Thus $P_{\pi,\alpha}$ are polynomials in $z$, and then $P_\pi(z,w)$ is a polynomial in $z$ and $\bar{w}$. Moreover, $\deg_z P_\pi(z,w)\leq N$, $\deg P_{\pi,\alpha}(z)\leq N$. 
		 This completes the proof.
	\end{proof}
	
	Let $\mathbb{P}_n$ denote the $n$-dimensional complex projective space. Then $\cn$ can be viewed as a subset of $\mathbb{P}_n$ via the natural embedding 
	\[
	\cn\to\mathbb{P}_n,\quad(z_1,\cdots,z_n)\mapsto[1,z_1,\cdots,z_n].
	\]
	For any algebraic set $A\subset\cn$, its closure $\overline{A}$ in $\mathbb{P}_n$ is an analytic subset of $\mathbb{P}_n$. For a $p$-dimensional linear space $L\subset\cn$, $\overline{L}$ is a $p$-dimensional linear space in $\mathbb{P}_n$. Its orthogonal complement in $\mathbb{P}_n$ is of dimension $n-p-1$.
	\[
	\overline{L}^\perp=\{[z_0,\cdots,z_n]:~(z_0,\cdots,z_n)\perp(1,w), w\in L\}.
	\]
	We have the following lemma \cite[7.3]{Chi}.
	
	\begin{lem} \label{lem: nonintersection implies proper}
		Let $A$ be a pure $p$-dimensional projective algebraic set in $\mathbb{P}_n$
		and let $L\subset\mathbb{P}_n$ be a complex $n-p-1$-dimensional plane not intersecting $A$. Then the projection $\pi: A\to L^\perp$ is proper.
	\end{lem}
	
	The following theorem will be used in the proof of Theorem \ref{thm: Smooth implies ASD}.
	\begin{thm}\label{thm: local CDF}
		Suppose $A\subset\cn$ is an algebraic set of pure dimension $p<n$ and $0\in A$. Assume the following.
		\begin{itemize}
			\item[(1)] Denote $L=\{(z',0):~z'\in\CC^p\}\subset\cn$ and suppose $\overline{L}^\perp\cap\overline{A}=\emptyset$. Therefore if we denote $\pi:\cn\to\CC^p$, $z\mapsto z':=(z_1,\cdots,z_p)$, then $\pi|_A$ is proper.
			\item[(2)] $U=U'\times U''$, where $0\in U'\subset\CC^p$ and $0\in U''\subset\CC^{n-p}$ are open sets. $\pi|_{A\cap U}$ is also proper. Moreover, $\pi^{-1}(0)\cap A\cap U=\{0\}$.
		\end{itemize}
		Then there exist open sets $0\in U_1'\subset U'$, $0\in U_1''\subset U''$, an open neighborhood $\mathcal{U}$ of the $n\times n$ identity matrix $I_{n\times n}$ in $U(n)$, and an open set $\W\subset\CC^{n-p}$ with the following properties. Denote $U_1=U_1'\times U_1''$. For any $l\in\mathcal{U}$, the projections $\pi|_{l(A)}$ and $\pi|_{l(A)\cap U_1}$ are proper. Moreover, for any $w\in \W$, the function
		\[
		\psi_w(z)=\frac{P_{\pi|_{l(A)}}(z,w)}{P_{\pi|_{l(A)\cap U_1}}(z,w)}
		\]
		is non-vanishing and holomorphic in $U_1$.
	\end{thm}
	\begin{proof}
		Since $\pi|_A$ is proper, we know that $\pi^{-1}(0)\cap A$ consists of finitely many points. Suppose
		\[
		\big(\pi^{-1}(0)\cap A\big)\backslash\{(0,0)\}=\{(0,a_1),\cdots,(0,a_k)\},
		\]	
		where $a_i\in\CC^{n-p}$. Take an open set $\W\subset\CC^{n-p}$ whose closure is contained in the open set $\{w\in\CC^{n-p}: \la a_i,w\ra\neq0\}$. We can find open neighborhoods $a_i\in V_i\subset\CC^{n-p}$, $0\in U_1''\subset\CC^{n-p}$ such that 
		\[
		\la\lambda-z'',w\ra\neq0,\quad \forall i, \forall\lambda\in V_i, \forall z''\in U_1'', \forall w\in \W.
		\]
		We claim that there exist an open neighborhood $\mathcal{U}$ of $I_{n\times n}$ in $U(n)$ and an open set $0\in U_1'\subset U'$ with the following property. Denote $U_1=U_1'\times U_1''$. Then $\pi|_{l(A)}$ and $\pi|_{l(A)\cap U_1}$ are proper. Moreover,
		\[
		\pi^{-1}(U_1')\cap l(A)\subset U_1\cup\bigg(\bigcup_{i=1}^kU_1'\times V_i\bigg).
		\]
		By assumption, $\overline{L}^\perp\cap\overline{A}=\emptyset$. Therefore we can find an open neighborhood $\mathcal{U}$ of $I_{n\times n}$ in $U(n)$ such that $\forall l\in\mathcal{U}$, $\overline{l^{-1}(L)}^\perp\cap\overline{A}=\emptyset$. Thus the projection from $A$ onto $l^{-1}(L)$ is proper. Equivalently, $\pi|_{l(A)}$ is proper. Also, by the proof of \cite[Theorem 7.4.2]{Chi}, $A$ is contained in the union of a ball $B:=\{z:~|z|<R\}$ and a cone $K:=\{z: |z''|<C|z'|\}$.
		
		On the other hand, by \cite[Corollary 4.2]{Chi}, we can take $U_2'$ small enough so that
		\[
		\pi^{-1}(U_2')\cap A\subset \big(U_2'\times U_1''\big)\cup\bigg(\bigcup_{i=1}^kU_2'\times V_i\bigg).
		\]
		If we shrink $\mathcal{U}$, we can ensure that $\forall l\in\mathcal{U}$, $l^{-1}\pi^{-1}(U_2')\backslash B$ is outside the cone $K$. Then $l^{-1}\pi^{-1}(U_2')\cap A\subset B$. Then if we shrink $\mathcal{U}$ again, we can find $0\in U_1'\subset U_2'$ so that $\forall l\in\mathcal{U}$, $l^{-1}\pi^{-1}(U_1')\cap B\subset \pi^{-1}(U_2')$. Thus 
		\[
		l^{-1}\pi^{-1}(U_1')\cap A=l^{-1}\pi^{-1}(U_1')\cap B\cap A\subset \pi^{-1}(U_2')\cap A\subset \big(U_2'\times U_1''\big)\cup\bigg(\bigcup_{i=1}^kU_2'\times V_i\bigg).
		\]
		If we replace the right hand side with a compact neighborhood $\mathcal{N}$ of $A\cap \pi^{-1}(U_1')$ contained in $\big(U_2'\times U_1''\big)\cup\bigg(\bigcup_{i=1}^kU_2'\times V_i\bigg)$, then the same method will give us $l^{-1}\pi^{-1}(U_1')\cap A\subset\mathcal{N}$. Then we can shrink $\mathcal{U}$ again to ensure $l(\mathcal{N})\subset \big(U_2'\times U_1''\big)\cup\bigg(\bigcup_{i=1}^kU_2'\times V_i\bigg)$. Then we have $\forall l\in\mathcal{U}$,
		\[
		\pi^{-1}(U_1')\cap l(A)\subset \big(U_2'\times U_1''\big)\cup\bigg(\bigcup_{i=1}^kU_2'\times V_i\bigg).
		\]
		Then obviously, 
		\[
		\pi^{-1}(U_1')\cap l(A)\subset \big(U_1'\times U_1''\big)\cup\bigg(\bigcup_{i=1}^kU_1'\times V_i\bigg).
		\]
		This proves our claim.
		
		The open sets $U_1''$ and $V_i$ can be chosen to be disjoint. By \cite[3.1 (3)]{Chi}, $\pi|_{l(A)\cap U_1}$ is also proper. Let $z'\in U_1'$ be outside the critical sets of both projections. Then 
		\[
		\pi^{-1}(z')\cap l(A)=\{(z',b_1(z')),\cdots,(z',b_l(z')); (z',a_1(z')),\cdots,(z', a_k(z'))\}.
		\]
		Here $b_j(z')\in U_1''$, $a_j(z')\in\cup_i V_i$. By definition, for $z\in U_1$ and $w\in \W$,
		\[
		P_{\pi|_{l(A)}}(z,w)=\Pi_{j=1}^l\la z''-b_j(z'),w\ra\times\Pi_{j=1}^k\la z''-a_j(z'),w\ra= P_{\pi|_{l(A)\cap U_1}}(z,w)\times\psi_w(z),
		\]
		where
		\[
		\psi_w(z)=\Pi_{j=1}^k\la z''-a_j(z'),w\ra.
		\]
		From our construction, it is straightforward that $\psi_w$ is non-vanishing on $U_1$. This completes the proof.	
	\end{proof}
	
	\begin{thm}\label{thm: proper projections on tangent plane}
		Suppose $A$ is a $p$-dimensional irreducible affine algebraic set in $\cn$ and $0\in A$. Assume that $A\nsubseteq e_n^\perp$. Let $$\mathcal{G}_n=\{L\in G(p,n): e_n\in L\}$$ and $$\tilde{\mathcal{G}}_n=\{L\in \mathcal{G}_n: \overline{L}^\perp\cap\overline{A}=\emptyset\}.$$
		Here $G(p,n)$ is the Grassmannian.
		Then $\tilde{\mathcal{G}}_n$ is a dense open set in $\mathcal{G}_n$. Equivalently, let $$\tilde{\mathcal{U}}_n=\{l\in U(n-1): \mathrm{span}\{l(e_1),\cdots,l(e_{p-1}),e_n\}\in\tilde{\mathcal{G}}_n\}.$$
		Then $\tilde{\mathcal{U}}_n$ is dense in $U(n-1)$.
	\end{thm}
	\begin{proof}
		The two statements are clearly equivalent. Let us prove the first statement. Consider the canonical projection. 
		\[
		\Pi: \CC^{n+1}\backslash\{0\}\to\mathbb{P}_n,\quad (z_0,z_1,\cdots,z_n)\mapsto [z_0,z_1,\cdots,z_n].
		\]
		For $L\in G(p,n)$, denote $\tilde{L}=\Pi^{-1}(\overline{L}^\perp)\cup\{0\}$ and $\tilde{A}=\Pi^{-1}(\overline{A})\cup\{0\}$. Then $\overline{L}^\perp\cap\overline{A}=\emptyset$ if and only if $\tilde{L}\cap\tilde{A}=\{0\}$. Since $A$ is irreducible and has dimension $p$, $\tilde{A}$ is a homogeneous irreducible algebraic set of dimension $p+1$. $\tilde{L}$ is a linear subspace of dimension $n-p$. The condition that $e_n\in L$ is equivalent to that $\tilde{L}\subset L_n:=e_n^\perp\subset\CC^{n+1}$. Let $\tilde{A}_n:=\tilde{A}\cap L_n$. Then $L\in\tilde{\mathcal{G}}_n$ if and only if $\tilde{L}\subset L_n$ and $\tilde{L}\cap\tilde{A}_n=\{0\}$.
		
		We claim that $\dim \tilde{A}_n\leq p$. Otherwise, $\dim \tilde{A}_n=p+1$. Since $\tilde{A}$ is irreducible, it cannot properly contain any algebraic set of the same dimension. So $\tilde{A}=\tilde{A}_n$ and therefore $\tilde{A}\subset L_n$. However, this implies that $A=\overline{A}\cap \cn\subset\Pi(\tilde{A}\backslash\{0\})\cap\cn\subset\Pi(L_n\backslash\{0\})\cap\cn=e_n^\perp$. A contradiction. Thus $\dim\tilde{A}\leq p$.
		
		Assume $\tilde{L}\subset L_n$. Both $\tilde{L}$ and $\tilde{A}_n$ are homogeneous varieties in $L_n\cong\CC^n$.  Thus $\tilde{L}\cap\tilde{A}_n=\{0\}$ if and only if their preimages in $\mathbb{P}_{n-1}$ do not intersect. The preimages of the two varieties have dimension $n-p-1$, $\leq p-1$, respectively. Thus the set of $\tilde{L}\subset L_n$ not intersecting $\tilde{A}$ form a dense open sent in $G(n-p,n)$. From this it is easy to see that $\tilde{\mathcal{G}}_n$ is a dense open set in $\mathcal{G}_n$. This completes our proof.

	\end{proof}

\end{document}